\documentclass[a4paper,reqno]{amsart}

\usepackage[T1]{fontenc}
\usepackage[utf8x]{inputenc}
\usepackage[english]{babel}
\usepackage{yfonts}
\usepackage{dsfont}
\usepackage{amscd,amssymb,amsmath,amsthm,amsfonts}
\usepackage{mathtools}
\usepackage{accents}
\usepackage{tensor}
\usepackage{graphicx}
\usepackage{subfigure}
\usepackage{tikz-cd}
\usetikzlibrary{calc,matrix,arrows,decorations.pathmorphing}
\usepackage{calc}
\usepackage[cal=boondox,scr=boondoxo]{mathalfa}
\usepackage{enumitem}
\usepackage{marginnote}
\usepackage{booktabs}
\usepackage{scalerel}[2016/12/29]
\usepackage{hyperref}

\usepackage{hyperref}
\hypersetup{colorlinks=true,pageanchor=false,
linkcolor=blue,citecolor=red,urlcolor=red}
\usepackage[all]{hypcap}
\usepackage{cleveref}
\usepackage{appendix}
\usepackage{soul}

\allowdisplaybreaks

\newtheorem{theorem}{Theorem}
\newtheorem{corollary}[theorem]{Corollary}
\newtheorem{proposition}[theorem]{Proposition}
\newtheorem{lemma}[theorem]{Lemma}
\newtheorem{problem}[theorem]{Problem}

\theoremstyle{definition}
\newtheorem{definition}[theorem]{Definition}

\theoremstyle{remark}

\newcommand{\Z}{\mathbb{Z}}

\newcommand{\R}{\mathbb{R}}
\newcommand{\C}{\mathbb{C}}
\newcommand{\Heis}{\mathcal{H}}

\newcommand{\id}{\mathrm{id}}

\newcommand{\End}{\mathrm{End}}

\newcommand{\Cob}{\mathrm{Cob}}

\makeatletter
\newcommand{\subalign}[1]{
  \vcenter{
    \Let@ \restore@math@cr \default@tag
    \baselineskip\fontdimen10 \scriptfont\tw@
    \advance\baselineskip\fontdimen12 \scriptfont\tw@
    \lineskip\thr@@\fontdimen8 \scriptfont\thr@@
    \lineskiplimit\lineskip
    \ialign{\hfil$\m@th\scriptstyle##$&$\m@th\scriptstyle{}##$\crcr
      #1\crcr
    }
  }
}
\makeatother

\def\clap#1{\hbox to 0pt{\hss#1\hss}}

\newcommand\doi[2]{\href{http://doi.org/#1}{#2}}

\newcommand{\CobL}{\mathrm{3Cob^{LC}}}
\newcommand{\bS}{\mathbb{S}}
\newcommand{\Ker}{\mathrm{Ker}} 
  
\begin{document}

\raggedbottom

\title{Abelian TQFTS and Schr\"odinger local systems}

\author[A. Andreev]{Aleksei Andreev}
\address{Institut f\"ur Mathematik,
Universit\"at Z\"urich,
Winterthurerstrasse 190,
CH-8057 Z\"urich.}
\email{aleksei.andreev@math.uzh.ch} 

\author[A. Beliakova]{Anna Beliakova} 
\address{Institut f\"ur Mathematik,
Universit\"at Z\"urich,
Winterthurerstrasse 190,
CH-8057 Z\"urich.}
\email{anna@math.uzh.ch}

\author[C. Blanchet]{Christian Blanchet} 
\address{Universit{\'e} Paris Cité and Sorbonne Universit{\'e}, CNRS, IMJ-PRG, F-75006 Paris, France} 
\email{christian.blanchet@imj-prg.fr}

\begin{abstract} In this paper  
we  construct an action of 3-cobordisms  on the finite dimensional Schr\"odinger 
representations of the  Heisenberg group
  by Lagrangian correspondences. 
  In addition,
we review the construction of the 
abelian Topological Quantum Field Theory (TQFT) associated with  a $q$-deformation of $U(1)$ for any root of unity $q$.
We prove that for
 3-cobor\-disms compatible with Lagrangian correspondences, 
  there is a normalization of the associated 
 Schr\"odinger bimodule action that reproduces the abelian TQFT. 

The full abelian TQFT provides a projective representation of the mapping class group $\mathrm{Mod}(\Sigma)$ on the Schrödinger representation,
which is linearizable at odd root of 1. Motivated by homo\-logy of surface configurations with Schrödinger representation as local coefficients, we define another 
projective action of $\mathrm{Mod}(\Sigma)$ on Schrödinger representations. We show that the latter is not linearizable by identifying the associated 2-cocycle.

\end{abstract}

\maketitle

  \hspace{8cm}{\em To the memory of Vaughan Jones,}

\vspace{-0.5mm}\hspace{8cm}{\em  the founder of quantum topology.}

\vspace{.5cm}

\section{Introduction}
The discovery of the Jones polynomial revolutionized low--dimensional topology.  The  new link invariants constructed by Jones, Kauffman,
HOMFLY-PT, Reshetikhin--Turaev etc. were extended to  
mapping class group representations, later shown to be asymptotically faithful, and  to  3-manifold invariants. These developments have reached their peak in  constructions of Topological Quantum Field Theories (TQFTs) 
\cite{Tu, BHMV}. The scope of ideas initiated by 
Vaughan Jones built the foundations
  for the new domain of mathematics -- the quantum topology.
One of the main open problems in quantum topology is to understand the topological nature of quantum invariants.

In the 90s, 
Lawrence \cite{Lawrence1990}   initiated a program aiming at homological 
interpretation of quantum invariants.  In 2001 Bigelow 
\cite{BigelowJones}
was able to read the Jones polynomial
from the intersection pairing on the twisted homology of the  configuration space $\mathrm{Conf}_n(\mathbb D^2_m)$ of $n$ points in
 $m$-punctured disc $\mathbb D^2_m$. This construction led to 
a   family  of  representations (indexed by $n$)
of the braid group $B_m$, that recovers for $n=1$ the Burau representation.
A spectacular achievement was the proof by Bigelow \cite{Bigelow2001} and Krammer \cite{Krammer2002} that
this braid group representation for $n=2$  is faithful,  showing the
 linearity of the braid group.
Bigelow's construction was extended later  to other quantum link invariants (see \cite{Bigelow2007, CristinaColjones, CristinaColalex, Martel} and references thereof).

Recently  homological mapping class group representations were constructed by
the third author together with Palmer and Shaukat \cite{HeisenbergHomology}.
The idea here was to use a {\it Heisenberg} cover  
of the space ${\mathrm{Conf}}_n(\Sigma)$
of unordered $n$ configurations in a surface $\Sigma$ with one boundary component, whose
group of deck transformations is the {\em Heisenberg group} $\Heis(\Sigma)$.
Recall that 
$\Heis(\Sigma)=\Z\times H_1(\Sigma, \Z)$ has the group law
$$ (k,x)(l,y)=(k+l+\,x.y,x+y)$$
where $x.y$ is the intersection pairing.
In detail, 
the surface braid group  $B_n(\Sigma):=\pi_1(\mathrm{Conf}_n(\Sigma))$ surjects onto $\Heis(\Sigma)$ (see Section \ref{sec.4.1}). The kernel of this map is a characteristic normal subgroup
of the surface braid group  by \cite[Prop.8]{HeisenbergHomology}, which determines
 the Heisenberg cover $\widetilde{\mathrm{Conf}}_n(\Sigma)$.

Since the group of  deck transformations $\Heis(\Sigma)$
 acts  on the chain groups of the Heisenberg cover,
any $\Heis(\Sigma)$-modu\-le $M$  can be used 
to construct a  twisted homology 
as follows.  We first extend 
the  action of the group to its group algebra
$ \C[\Heis]\to \End(M)$
by linearity and then construct
 a complex $$C_\bullet\left(\widetilde{\mathrm{Conf}}_n(\Sigma)\otimes_{\C[\Heis(\Sigma)]} M\right).$$ 
Its homology, known 
as twisted homology of $\mathrm{Conf}_n(\Sigma)$ with coefficients in $M$, is denoted by $H_\bullet(\mathrm{Conf}_n(\Sigma),M)$
(compare \cite[Ch.  5]{Davis}, \cite[Ch.  3.H]{Hatcher}).
An interesting choice of $M$ provides a finite dimensional 
Schr\"odinger representation $W_q(L)$
of a finite quotient 
of $\Heis(\Sigma)$,
which depends on a choice of
a Lagrangian $L\subset H_1(\Sigma, \Z)$ and a root of unity $q$.
If the order of $q$ is odd, the resulting mapping class group representations were recently 
shown to contain the 
quantum representations arising from the non-semisimple TQFT for the small quantum
$\mathfrak{sl}_2$  by De Renzi and Martel \cite{dRM}. 
In particular, they defined the action of the quantum $\mathfrak{sl}_2$
on the Schr\"odinger homology explicitly and showed that it commutes with the action of the
mapping class group.

To complete Lawrence--Bigelow program 
we are lacking  homological interpretation of quantum 3-manifold invariants
and of the action of 3-cobordisms on  Schr\"odinger homologies. This paper is a first step in this direction. 
Here we construct an action of 3-cobordisms on Schr\"odinger 
representations
 by Lagrangian correspondences.
In addition, we show that
 on a certain subcategory of extended 3-cobordisms and after a suitable normalization
this action  recovers the abelian TQFT.

Abelian TQFTs are  functorial extensions of   $3$-manifold
invariants constructed by Murakami--Ohtsuki--Okada from linking matrices \cite{MOO}.
Their connections with theta functions and Schr\"odinger representations, 
in the case when the quantum parameter (called $t$ in these papers) is a  root of unity of order divisible by $4$, 
were extensively studied by Gelca and collaborators
\cite{Gelca_Uribe,Gelca_Hamilton,Gelca_Hamilton2,Gelca}.
 Here we work with an arbitrary root of unity. We show that
  interesting cases 
 are if the order is either odd or divisible by $4$.
In the latter case
we complete the work of
Gelca and al.  by constructing  TQFTs via
{\it modularization functor}. In addition, we discuss refined TQFTs
corresponding to the choice of a spin structure or a first cohomology class
on 3-manifolds.

Our preferred cobordism category is the Crane--Yetter category
$3\Cob$ of connected oriented 3-cobordisms between connected surfaces with one boundary component and with the boun\-dary connected sum as the monoidal structure. This category has a
beautiful algebraic presentation: it is monoidally gene\-rated by the 
Hopf algebra object --- the torus with one boundary component \cite{BBDP, BobtchevaPiergallini}.
By the result of \cite{BD}, for any
  finite unimodular  ribbon category $\mathcal C$, there exists a monoidal
 TQFT functor from $3\Cob^\sigma$ to $\mathcal{C}$ determined  by sending the torus with one boundary component to the terminal object  of $\mathcal C$, called the {\it end}.
 Here $3\Cob^\sigma$ is the category of {\it extended}
 3-cobordisms, whose objects are connected surfaces with one boundary component equipped with a choice of Lagrangian, and morphisms are 3-cobordisms equipped with  natural numbers 
 called weights.
 The composition includes a correction term given by a Maslov index. 
 Note that if $\mathcal C$ is the category of modules over a unimodular ribbon Hopf algebra $H$, then the {\it end} of $\mathcal C$ is the adjoint 
 representation $(H, \rhd)$ 
where $\rhd$ denotes the adjoint action. 

Let us define a subcategory $3\Cob^{\rm{LC}}$ of  $3\Cob^\sigma$ having the same
objects, but a smaller set of morphisms.
A cobordism $C=(C,0)$ 
belongs to $3\Cob^{\rm{LC}}\!\!\left((\Sigma_-,L_-),(\Sigma_+,L_+)\right)$ if and only if   $L_+=L_C.L_-$ where
$$L_C.L_-=\{y\in H_1(\Sigma_+)\, | \, \exists \, x\in L_-,\  (x,y)\in L_C\}
\quad \text{and}$$
$$L_C=\mathrm{Ker}\left(i_*: H_1(\partial C, \Z)=H_1(-\Sigma_-,\Z)\oplus H_1(\Sigma_+,\Z) \to H_1 (C, \Z)\right) .$$
We say that $L_+$ is determined by the {\it Lagrangian correspondence} given by $C$. 
 
In subcategory  $3\Cob^{\rm{LC}}$ all anomalies vanish. We get a linear representation of the subgroup of the mapping class group fixing a Lagrangian. The full mapping class group is replaced by a groupoid whose objects are Lagrangians  and morphisms are compatible mapping classes. This {\em action groupoid} is a subcategory in $3\Cob^{\rm{LC}}$.

Assume $q\in \C$ is a primitive $p$-th root of unity of order 
 $p\geq 3$ and $p\not\equiv 2\pmod 4$. Let $p'=p$ if $p$ is odd, and $p'=p/2$ otherwise.
 We define the finite  Heisenberg group $\Heis_p(\Sigma)$
as a quotient of $\Heis(\Sigma)$ by the 
normal subgroup
$$I_p:=\left\{(pk, p'x)\;|\; k\in \Z, x\in H_1(\Sigma, \Z)\right\}.$$
The group $\Heis_p(\Sigma)$ is isomorphic to a $\Z_p$-extension of
$H_1(\Sigma,\Z_{p'})$, where
 we use the shorthand $\Z_p$ for  $\Z/p\Z$ (see Section
\ref{sec3.2} for more details).

For
a given  Lagrangian submodule $L\subset H_1(\Sigma,\Z)$, let $L_p=L\otimes \Z_{p'}\subset  H_1(\Sigma,\Z_{p'})$ 
and $\widetilde L_p=\Z_p\times L_p\subset \Heis_p(\Sigma)$ be a maximal abelian subgroup. 
\reversemarginpar
Denote by $\C_q$ a  $1$-dimensional representation of $\widetilde L_p$, where $(k,x)$ acts by $q^k$.  Then  inducing from 
$\C_q$ we obtain
$$W_q(L)=\C[\Heis_p(\Sigma)]\otimes_{\C[\widetilde L_p]} \C_q$$
 a ${p'}^g$-dimensional {\em Schr\"odinger} representation of 
  $\Heis_p(\Sigma)$. Note that as a $\C[\Heis_p(\Sigma)]$-module $W_q(L)$ is generated by $\mathbf{1} \in \C_q$. 
Given a cobordism in the category $3\Cob^{\rm{LC}}$, $C :(\Sigma_-,L_-)\rightarrow (\Sigma_+,L_+)$, with $L_+=L_C.L_-$, we have a Schr\"odinger representation $W(L_C)$ of the Heisenberg group $\Heis(\partial C)$ which can be considered as a 
$(\C[\Heis(\Sigma_+)],\C[\Heis(\Sigma_-)])$-bimodule, after identifying
the subgroup $\Heis(-\Sigma_-)\subset \Heis(\partial C)$  with $\Heis(\Sigma_-)^{op}$, and defining a right action of $\Heis(\Sigma_-)$ on $W_q(L_C)$   as the left action of the same element of $\Heis(-\Sigma_-)$. 

The main results of this paper can be formulated as follows.


\begin{theorem}\label{thm:main}
Assume $p\not\equiv 2\pmod 4$.
 For any cobordism $C$  from $(\Sigma_-,L_-)$ to
$(\Sigma_+,L_+)$ in $3\Cob^{\rm{LC}}$ 
 there exists an  isomorphism of $\Z[\Heis(\Sigma_+)]$-modules
 \[\psi_C: W_q(L_C)\otimes_{\C[\Heis_p(\Sigma_-)]} W_q(L_-)\xrightarrow{\sim}
  W_q(L_+)
  \]  sending $\mathbf{1} \otimes \mathbf{1}$ to $\mathbf{1}$.
  In addition, 
 the map
\begin{align*}
 F_C: W_q(L_-)& \to W_q(L_+)\\
 w& \mapsto  \psi_C(\mathbf{1}\otimes w).
\end{align*} 
   defines  a monoidal functor
$F:3\Cob^{\rm{LC}}\to \mathrm{Vect}_\C$ which associates with an 
object $(\Sigma,L)$ the finite dimensional Schr\"odinger representation
$W_q(L)$. 

\end{theorem}


The proof 
 uses a modification of the  Juhasz's presentation 
of cobordisms categories, which works for $3\Cob^{\rm{LC}}$
and is presented in Appendix. 
Throughout this paper we  refer to the functor $F$ as {\it Schr\"odinger} TQFT. 

Our second result compares the abelian and Schr\"odinger TQFTs. In particular, we will show that  TQFT maps for a given cobordism $C$
coincide up to a normalization.
 The normalising coefficient, which we  denote by $Z(\check C)$,  is actually 
 the Murakami--Ohtsuki--Okada invariant of a closed 3-manifold $\check C$
 obtained from $C$ by gluing of two standard handlebodies $(H_\pm, L_\pm)$ with
 $\partial H_\pm=\Sigma_\pm$ and $L_\pm$ 
 generated by meridians, along diffeomorphisms identifying the Lagrangians. This leads us to the following statement.
 
\begin{theorem}\label{thm:iso-tqft}
The monoidal functor  $\check F:3\Cob^{\rm{LC}}\to \mathrm{Vect}_\C$
sending a cobordism $C$ to
\begin{align*}
\check F_C: W_q(L_-)& \to W_q(L_+)\\
 w& \mapsto Z(\check C)\, \psi_C(\mathbf{1}\otimes w)
 \end{align*}
  coincides with the abelian TQFT at $q$ restricted to 
$3\Cob^{\rm{LC}}$.
\end{theorem}

Observe that the normalization coefficient $Z(\check C)$ is equal to zero if and only if
there exists $\alpha \in H^1(\check C, \Z_{p'})$ with 
non zero triple product $\alpha\cup \alpha\cup \alpha$ \cite[Thm. 3.2]{MOO}, however 
the  Schr\"odinger action is always non vanishing.


Finally, in Section  \ref{Sec4}
 we study the projective action of the full
mapping class group $\mathrm{Mod}(\Sigma)$ on the  Schr\"odinger representations.
We use here the Stone-von Neumann theorem to identify Schr\"odinger representations for different Lagrangians.
The symplectic action sends $f\in \mathrm{Mod}(\Sigma)$ to 
the automorphim $(k,x) \mapsto (k, f_*(x))$ of the Heisenberg group.
We use this automorphism to build a projective action 
of $\mathrm{Mod}(\Sigma)$
on the Schr\" odinger representation $W_q(L)$ known as Weil representation.
However, if we ask  the natural $\mathrm{Mod}(\Sigma)$ action 
on the surface braid group $B_n(\Sigma)$
to commute with the projection to $\Heis(\Sigma)$ we get a different automorphism
$$f_\Heis(k,x) = (k+\theta_f(x), f_*(x))$$ with $f\mapsto \theta_f\in 
\mathrm{Hom}(H_1(\Sigma), \Z)$  a crossed homomorphism.
For odd $p$,  $f_\Heis$ 
 can be used to construct another projective action on
the Schr\"odinger representations. This action is actually compatible with the corresponding local systems on $\mathrm{Conf}_n(\Sigma)$.
Our analysis shows that in the odd case the symplectic action on Schr\"odinger  representations 
 is linearizable, however the latter action intertwining $f_\Heis$  does not.

We plan to use these results to construct an action of
cobordisms on 
the homology of  $\mathrm{Conf}_n(\Sigma)$ twisted by
Schr\"odinger representations and provide a homological interpretation of the Kerler--Lyubashenko TQFTs. Our long term goal will be to use infinite
dimensional Schr\"odinger representations to construct TQFTs with  generic 
quantum parameter $q$,
rather than at a root of unity. An existence of such TQFTs was predicted by physicists. They are expected to play a crucial role in the categorification
of quantum 3-manifold invariants \cite{gukov-vafa}. Lagrangian Floer homology may serve as an inspiration for this purpose.

The paper is organized as follows. In Section \ref{AbTQFT} we review
representation theoretical and skein constructions of abelian  TQFTs,
we discuss
   modularization functors, refinements
   as well as the action of the mapping class group and its extensions.
    In Section \ref{sec:proofs} we prove the two main theorems.
In Section \ref{Sec4} we define the two projective mapping class group actions on the Schr\"odinger representations and study the associated 2-cocycles.
Juhasz construction for lagrangian cobordisms is recalled in Appendix. 

\section{Abelian TQFTs}
\label{AbTQFT}
\subsection{Algebraic approach} 
 Let $q\in \mathbb S^1\subset \C$ be a primitive $p$-th root of $1$ 
 and $p\geq 3$ is an 
 integer.
 Let $p'=p$ if $p$ is odd and 
  $p'=p/2$ if $p$ even.
  Consider the group algebra $H=\C[K]/(K^{p}-1)$ of the cyclic group. 
  This algebra
can be identified with the Cartan part of the quantum $\mathfrak{sl}_2$ 
at $q$ by
extending the group monomorphism
\begin{align*}
 U(1)&\to SL(2, \C)\\
z&\to \left(\begin{matrix}  z&0\\
 0& \bar{z} \end{matrix}\right)
\end{align*}
 For this reason,   abelian
TQFTs are also called $U(1)$ TQFTs.

 The algebra $H$ has a natural Hopf algebra structure with a grouplike
 generator, i.e. $\Delta (K)=K\otimes K$, $S(K)=K^{-1}$. 
 Moreover, $H$ is a ribbon Hopf algebra with $R$-matrix and its inverse
  given by
\[ R=\frac{1}{p}\sum_{0\leq i,j\leq p-1}
q^{-ij} K^i\otimes K^j,\quad\quad
R^{-1}=\frac{1}{p}\sum_{0\leq i,j\leq p-1}
q^{ij} K^{-i}\otimes K^{-j},\]
 the ribbon elements
\[v=\frac{1}{p}\sum_{0\leq i,j\leq p-1} q^{i(j-i)}K^j, \quad\quad
v^{-1}=\frac{1}{p}\sum_{0\leq i,j\leq p-1} q^{i(i-j)}K^{-j}
\]  
  and the trivial pivotal structure.

Similarly to the $U_q(\mathfrak{sl}_2)$ case, the representation 
category $H\mathrm{-mod}$ has $p$ simple modules $V_k$ for $0\leq k\leq p-1$.
However, here $V_k$ is the $1$-dimensional representation determined by its character $K\mapsto q^k$. Also in our case, the fusion rules are very simple: $V_i \otimes V_j=V_{i+j}$ where the index
$i+j$ is taken modulo $p$. Hence,
   all objects $V_j$ are {\it invertible}, meaning
 that for each $j$ there exists $k=p-j$ such that 
 $V_j \otimes V_k=V_0$, where $V_0$  is the tensor unit of  $H\mathrm{-mod}$.
 The $R$-matrix is  acting by $q^{kl}$ on $V_k\otimes V_l$.
 
 \subsection{Skein approach}  \label{ss:skein}
For  explicit computations, it is more convenient to work with a skein theoretic construction.

Consider the   skein relations depicted in Figure \ref{skein}.
\begin{figure}
\centering
\begin{subfigure}{\includegraphics[scale=0.4]{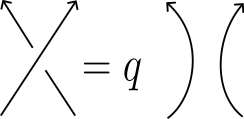}}\end{subfigure}\hspace{1.7cm}
\begin{subfigure}{\includegraphics[scale=0.4]{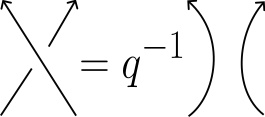}}\end{subfigure}\hspace{1.7cm}
\raisebox{2mm}{\begin{subfigure}{\includegraphics[scale=0.4]{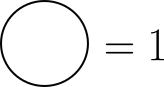}}\end{subfigure}}
\caption{$U(1)$ skein relations.}
\label{skein}
\end{figure}
Given a 3-manifold $M$, a skein module $S(M)$ 
is a $\C$-vector space generated by framed links in $M$  modulo the skein relations. For a surface $F$ it is customary to denote by $S(F)$ the skein algebra $S(F\times [0,1])$.
 We will usually identify a coloring of a  component $K$
 of a framed link with an element
of the  skein algebra $S(A)$, where the annulus  $ A$ is embedded along $K$ by using the framing.
For example,
$V_j$-coloring is represented by an element $y^j$
where $y$ is the core of  $A$. Here we use the usual 
algebra structure on $S(A)$ to identify
$y^j$ with $j$ parallel copies of $y$.
The Kirby color is $$\Omega\equiv\sum_{k=0}^{p-1}y^k \in S(A).$$ 

The  $\Omega$-colored $(+1)$-framed unknot gets value
\begin{equation}\label{eq:anomaly}
G\equiv\sum^{p-1}_{k=0} q^{k^2}=\left\{\begin{array}{ll}
\varepsilon\sqrt{p}\,  (1+\sqrt{-1})& \text{ if $p\equiv  0$ (mod $4$)}\\
\pm \sqrt{p}&\text{ if $p\equiv  1$ (mod $4$)},\\
0& \text{ if $p\equiv  2$ (mod $4$)}\\
\pm \sqrt{-p}
&\text{ if $p\equiv  3$ (mod $4$)}
\end{array}\right.
\end{equation}
by the well-known Gauss formula, where $\varepsilon$ is a $4$-th root of 1.
 For $p=0\pmod 4$, we explain in the next section, why
  the sums for $G$ and $\Omega$ should be taken till $p'-1$,
  and we denote them by
\begin{equation}
\label{eq:gausssum}
g\equiv\sum^{p'-1}_{k=0} q^{k^2}\quad \text{and} \quad\omega\equiv\sum_{k=0}^{p'-1}y^k .
  \end{equation}  
  Recall that  for $p$ odd, $p' =p$ and $g=G$, $\omega=\Omega$. Define $\eta$ and $\kappa$ by $|g|=\eta^{-1}$ and $\kappa=\eta g$. Then $\eta^{-1}=\sqrt{p'}$ while $\kappa$ is an eighth root of unity for $p\equiv 0 \pmod 4$ and $\kappa \in \{±1, ±\sqrt{-1}\}$ for $p$ odd.
  Hence, for all $p$ except $p\equiv 2\pmod 4$, we can define
the invariant of a closed 3-manifold $M$ obtained by surgery on 
$ S^3$ along a framed $n$ component link $L$ as follows
\[ Z(M)= \kappa^{-\mathrm{sign}(L)}\eta
\langle \eta\omega,\dots, \eta\omega   \rangle_L
\]
where  $\mathrm{sign}(L)$ is the signature of the linking matrix
and $\langle x_1, \dots, x_n\rangle_L$ denotes the evaluation of $L$ whose $i$-th component is colored by $x_i$
 in the skein algebra $S(\R^3)$.

 The normalization is chosen in such a way that  $$Z(S^2\times S^1)= \eta\langle \eta\omega\rangle_{\text{0-framed unknot}}=\eta^2 \sum^{p'-1}_{k=0} \langle y^k\rangle_{\text{unknot}}= \eta^2 p' =1\quad \text{and} \quad Z(S^3)=\eta=1/\sqrt{p'}.$$


The right Dehn twist along a curve $\gamma$ is represented by coloring the curve $\gamma$ with
 $$\eta\omega_-=\eta\,\sum_{k=0}^{p'-1}q^{-k^2}y^k  ,$$
where $y$ is the core of $\gamma$.

 If $p\equiv 0\pmod 4$, then we split $\omega=\omega_0+\omega_1$
into even and odd colors. 
Then we define an  additional topological structure
on $M$ that determines a $\Z/2\Z$-grading on the components of $L$, and thus a $\Z/2\Z$-grading on their colorings.
  In particular, for $p\equiv 4\pmod 8$,
    we  construct an invariant of the pair $(M,s)$
 \[ Z(M,s)= \kappa^{-\mathrm{sign}(L)}\eta\langle \eta\omega_{s_1},\dots, \eta\omega_{s_n}   \rangle_L
\]
where $(s_1, \dots, s_n)\in (\Z/2\Z)^n$  satisfying 
$$\sum^n_{j=1} L_{ij}s_j=L_{ii} \pmod 2\quad$$
determines a characteristic sublink of $L$ corresponding to the spin structure $s$ on $M$. Analogously, for $p\equiv 0\pmod 8$ we  construct invariants of a pair $(M,h)$
\[ Z(M,h)= \kappa^{-\mathrm{sign}(L)}\eta\langle \eta\omega_{h_1},\dots, \eta\omega_{h_n}   \rangle_L
\]
where $(h_1, \dots, h_n)\in (\Z/2\Z)^n$  satisfying 
$$\sum^n_{j=1} L_{ij}h_j=0 \pmod 2\quad$$
determines the first cohomology class
$h\in H^1(M, \Z/2\Z)$. In both cases, 
$Z(M)$ is the sum of the refined invariants over all choices of the additional structure.



\subsection{Modularization and refinements} 
   A $\C$-linear ribbon category with a finite number of dominating simple objects
   is called {\it premodular}. If in addition, the monodromy $S$-matrix is invertible,
   then the category is modular.
In our case the $S$-matrix, whose $(i,j)$ component is the invariant $q^{2ij}$
of the $(i,j)$-colored Hopf link, 
   is invertible only for odd $p$, and in this case
 $H\mathrm{-mod}$ is {\it modular},
   providing an abelian TQFT by  standard constructions \cite{Tu} or
   \cite{BHMV}.   
   

We call a premodular category  $\mathcal C$
   {\it modularizable}, if
there  exists a 
braided monoidal essentially surjective
functor from $\mathcal C$  to a modular category, sending
the subcategory of transparent objects to the tensor unit.   
In \cite[Prop. 4.2]{Bru} Brugui\`eres gave a simple criterion for a premodular category to be  
 modularizable, see also \cite{Muger}. In particular, such category cannot contain {\it transparent} objects with twist coefficient $-1$.
 Recall that an object is called transparent,
 if it  has trivial braiding with any other object.
Observe that the row in the $S$-matrix corresponding to the transparent object is colinear with the one for the tensor unit.
 
 If $p$ is even, 
   $H\mathrm{-mod}$ is a  premodular 
   category.   
   The  object $V_{p'}$ is { transparent}  and has
    twist coefficient  $q^{p'^2}$, which is $1$ if $p'$ is even
   and $-1$ if $p'$ is odd. 
    Using results of 
\cite{Bru}, we deduce that
in the case when $p\equiv 0\pmod 4$, 
$H\mathrm{-mod}$ is modulari\-zable.
The resulting modular category has $p'$ simple objects, 
that are all invertible.
The new Kirby color is given in \eqref{eq:gausssum}.
Hence, we have $\eta=|g|^{-1}=(\sqrt{p'})^{-1}$ in all cases when invariant is defined. 

Furthermore, if $p\equiv 4\pmod 8$, the object $V_{p'/2}$ has twist coefficient $-1$. 
From \cite{BBC} 
 we deduce that  our category in this case is actually
   {\it spin modular},
   hence providing an abelian spin TQFT for 3-cobordisms equipped with a spin structure. 
   Analogously, if $p\equiv 0\pmod 8$,  
 we can construct a refined  TQFT that gives rise to invariants of 3-cobordisms equipped with first cohomology classes
   over $\Z/2\Z$. We refer to \cite{BBC} for details
   about the construction of the refined invariants and their properties.

   In the case $p\equiv 2\pmod 4$, $H\mathrm{-mod}$ is not
 modularizable.
The best we can do in this case to obtain 3-manifold invariants 
is to consider the degree $0$ subcategory  with respect
to the $\Z/2\Z$-grading  given by the action of $K^{p'}$.
The corresponding invariants will coincide with those obtained with the quantum 
 parameter of odd order equal to $p'$.

To construct a  map associated by an abelian TQFT with a  
3-cobordism $C:\Sigma_-\to\Sigma_+$, we first need to choose parametrizations of surfaces $\Sigma_\pm$, i.e.  diffeomorphisms $\phi_\pm: \Sigma_{g_\pm} \to \Sigma_\pm$ where $\Sigma_g$ is the standard 
genus $g$ surface.
If $p\not\equiv 2\pmod 4$,
the TQFT vector space associated with  $\Sigma_g$
 has dimension ${p'}^g$. 
 A basis $\{y^{\bf i}, {\bf i}=(i_1, \dots, i_g),
 0\leq i_j\leq p'-1\}$ is given by $p'$ colorings of $g$ cores of  the 1-handles of a bounding handlebody $H_g$.
The $({\bf i},{\bf j})$-matrix element of 
the TQFT map is constructed as follows: We 
glue the standard handlebodies $H_{g_\pm}$ 
to  $C$ along the parametrizations.
Inside $H_{g_-}$ we put the link $y^{\bf j}$ and inside $H_{g_+}$ the link $y^{\bf i}$.
The result is  a closed $3$-manifold $\check M=S^3(L)$ with a collection of circles $c^+\cup c^-$ inside, then
$$Z(C)^{\mathbf i}_{\mathbf j}:=  \kappa^{-\mathrm{sign}(L)} \eta^{g_+}\langle \eta\omega,\dots, \eta\omega, y^{\bf i}, y^{\bf j}   \rangle_{L\cup c_+\cup c_-}.$$
By using the universal construction \cite{BHMV}, this map can also 
be computed by gluing just one handlebody
$(H_{g_-}, y^{\bf j})$ to $C$ and by evaluating  the result in the skein
of $C\cup H_{g_-}$. The parametrization reduces in this approach to the choice of  Lagrangian $L\subset H_1(\Sigma, \Z)$, which is equal 
to $\ker: H_1(\Sigma, \Z)\to H_1(H_g, \Z)$, and its complement
$L^\vee$.
Since all curves representing elements of $L$ 
are trivial in the skein of $H_g$, the basis curves $y^{\mathbf j}$ of the TQFT vector space are parametrized by a basis of $L^\vee$.






In this paper we will be particularly interested in the Crane--Yetter category $3\Cob$ of connected 3-cobordisms between connected
surfaces with one boundary component. In this category the monoidal product is given by the boundary connected sum rather than by the disjoint union, thus leading 
to a rich algebraic structure \cite{BobtchevaPiergallini}.
By the result of \cite{BD}, 
the category $\mathrm{3Cob}^{\sigma}$ of extended 3-cobordisms is universal in the sense that
for any
  finite unimodular  ribbon category $\mathcal C$, there exists a 
 TQFT functor $F:3\Cob^\sigma\to \mathcal{C}$ defined  by sending the torus with one boundary component to the {\it end} of $\mathcal C$.
In our case, for odd $p$  
$$\mathit{end}(H\mathrm{-mod})=\oplus^{p-1}_{j=0}\,  V_j  .
$$
Modularization creates an isomorphism $V_k\cong V_k\otimes V_{p'}$, hence
for $p\equiv 0\pmod 4$ we have
$$\mathit{end}(H\mathrm{-mod})=\oplus^{p'-1}_{j=0}\,  V_j  .
$$
In both cases, the vector space associated by $F$ to a genus $g$
surface with one boundary component has dimension $p'^g$.

For even $p'$,  refined TQFTs on $3\Cob^\sigma$ can be constructed along the lines 
of \cite{BD1}. On the standard cobordism category this was done in  \cite{B, BM}.

\subsection{Extended cobordisms and 
Lagrangian correspondence}\label{sec:LC}

 Let us recall that the  skein or Reshetikhin--Turaev 
 TQFT constructions give rise to  projective representations
 of the mapping class group and the gluing formula has a 
 so-called {\it framing anomaly} which can be resolved by using
 {\it extended cobordisms}. The latter are given by  a pair: a 3-cobordism
 between surfaces equipped with  Lagrangian subspaces in the first
 homology group
 and a natural number. This approach leads to a representation
 of a certain central extension of the mapping class group.

If $p$ is odd, then the framing anomaly $\kappa$, defined as $\frac{g}{|g|}$, where the Gauss sum $g$ in \eqref{eq:anomaly}, 
  is a $4$-th root of $1$.
From 
\cite[Remark 6.9]{GilmerMasbaum} we can deduce
that 
the corresponding TQFT contains a native representation of the mapping class group.
This is because, the central generator of the extension acts
by $\kappa^4=1$, hence the index $4$ subgroup described in \cite{GilmerMasbaum} is the trivial extension.
Recall that the metaplectic group $\mathrm{Mp}_{2g}$ is the non trivial double cover of the symplectic group $\mathrm{Sp}_{2g}$.
The metaplectic mapping class group is the pull back of this double cover using the symplectic action.
In the case $p\equiv  0$ (mod $4$) the framing anomaly $\kappa$ is a primitive $8$-th root of unity and the above argument shows that the TQFT contains a native representation of the metaplectic mapping class group.


To avoid  anomaly issues in general,  we will work with a subcategory 
$3\Cob^{\rm{LC}}$ 
of the category of connected {\it extended} 3-cobordisms between connected
surfaces with one boundary component equipped with Lagrangians.
Objects of $3\Cob^{\rm{LC}}$ are as in $3\Cob^{\sigma}$, namely pairs: a
connected surface with one boundary component $\Sigma$ and
  a Lagrangian subspace $L\subset H_1(\Sigma,\Z)$. Recall that 
  a Lagrangian  is a maximal submodule
 with vanishing intersection pairing.
 For any 3-cobordism $C:\Sigma_- \rightarrow \Sigma_+$ we define a Lagrangian correspondence
 $$L_C=\mathrm{Ker}\left(i_*: H_1(\partial C, \Z)=H_1(-\Sigma_-,\Z)\oplus H_1(\Sigma_+,\Z) \to H_1 (C, \Z)\right) .$$
Now given Lagrangians $L_\pm \subset H_1(\Sigma_\pm, \Z)$ 
 the action of $L_C$ on $L_-$ is defined as follows:
 $$L_C.L_-=\{y\in H_1(\Sigma_+)\, | \, \exists \, x\in L_-,\  (x,y)\in L_C\}$$
 The pair $(C, 0)$ belongs to $3\Cob^{\rm{LC}}\!\!\left((\Sigma_-,L_-),(\Sigma_+,L_+)\right)$ if and only if   $L_C.L_-=L_+$.
 If we restrict to mapping cylinders we obtain the so called {\em action groupoid} of the mapping class group action on Lagrangian subspaces.

Restriction of the TQFT functor to $3\Cob^{\rm{LC}}$ kills
all Maslov indices needed to compute framing anomalies in gluing formulas (compare \cite[Sec.2]{GilmerMasbaum}).



\section{Proofs}\label{sec:proofs}
In this section  we  prove our two main results. We will always assume that  $p\geq 3$, $p\not \equiv  2 \pmod 4$ and $p'=p$ if $p$ is odd and 
  $p'=p/2$ if $p$ even.  

For a Lagrangian submodule $L\subset H_1(\Sigma,\Z)$ let $L^\vee$ be a complement of $L$. Then 
we set $L_p:=L\otimes \Z_{p'}$, and 
$L^\vee_p:=L^\vee\otimes \Z_{p'}$.
The finite quotient $\Heis_p(\Sigma){:=\Heis(\Sigma)/I_p}$ of the Heisenberg group, defined in Introduction,
 coincides with the semidirect product
$$ \Heis_p(\Sigma)\cong \left(\Z_p \times L_p\right) \ltimes L^\vee_p$$
where the multiplication  on the right hand side is given by the formula: $(k,a,b)(k',a',b')=(k+k'+2a.b',a+a',b+b')$. 
The isomorphism  between two models is {induced} by {the homomorphism} 
$${\Heis(\Sigma)\rightarrow\left(\Z_p \times L_p\right) \ltimes L^\vee_p:\:(k, a + b) \mapsto (k+a.b\pmod{p},\, a\pmod{p'},\,b\pmod{p'})}$$
 where $a\in L$, $b\in L^{\vee}$ (see \cite[Prop. 2.3]{Gelca_Uribe} for more details).

Using this isomorphism it is easy to check that 
{$\widetilde L_{p}=\Z_p\times L_p\subset \Heis_p(\Sigma)$}  is a maximal abelian subgroup. Let $q$ be a primitive $p$-th root of unity. Denote by $\C_q$ a  $1$-dimensional representation of {$\widetilde L_p$} , where $(k,x)$ acts by $q^k$. 
Then  inducing from 
$\C_q$ we define
$$W_q(L)=\C[\Heis_p(\Sigma)]\otimes_{\C[\widetilde L_p]} \C_q$$
 a ${p'}^g$-dimensional {\em Schr\"odinger representation} of 
the finite Heisenberg group  $\Heis_p(\Sigma)$.
Let us denote by $\mathbf 1$ the canonical generator of $W_q(L)$ as $\C[\Heis_p(\Sigma)]$-module.
Moreover, throughout this section 
to {\it simplify notation} we denote  $L\otimes \Z_{p'}$  by ${\sf L}$. 

Any Lagrangian ${L}^\vee\subset H_1(\Sigma,\Z)$ complementary to ${ L}$ provides  a  basis for $W_q(L)$ indexed by
${\sf L}^\vee$. 
Given $b\in {\sf L}^\vee$
we denote by $v_b$  the corresponding basis vector.
In this basis the left action of the Heisenberg group is as follows.
\begin{itemize}
\item The central generator $u=(1,0)$ acts by $v_b\mapsto q v_b$.
\item For  $y\in {L}^\vee$, $(0,y)$ acts by translation: $v_b\mapsto v_{b+y}$, where the index is modulo $p'$.
\item For  $x\in {L}$, $(0,x)$ acts by $v_b\mapsto q^{2x.b}\,v_b$,  {where $x.b$ is computed by lifting $b$ to any preimage in $H_1(\Sigma)$ and is well-defined since $q$ is a $p$-th root of unity}.
\end{itemize}
In the last step we used the rule $(0,x)(0,b)=(x.b, x+b)=(0,b)(2x.b, x)$.

Our  results provide a new model for the abelian TQFT 
based on Schr\"odinger representations. Let us prove our main theorems.


\begin{proof}[{\bf Proof of Theorem \ref{thm:main}}]
Let 
$C\in 3\Cob^{\rm{LC}}((\Sigma_-,L_-), (\Sigma_+,L_+))$.
Then we have three Heisenberg groups $\Heis(\Sigma_-)$, $\Heis(\Sigma_+)$ and $\Heis(\partial C)$, and respective  Schr\"odinger representations  $W_q(L_-)$, $W_q(L_+)$ and $W_q(L_C)$.

Using that $\partial C=-\Sigma_-\cup_{{S}^1} \Sigma_+$ and
the inclusions $H_1(-\Sigma_-,\Z)\rightarrow H_1(\partial C,\Z)$,  $H_1(\Sigma_+,\Z)\rightarrow H_1(\partial C,\Z)$
we have commuting actions of $\Heis(-\Sigma_-)$ and $\Heis(\Sigma_+)$ on $W_q(L_C)$. 
Actually,  $W_q(L_C)$ can be viewed as a $(\C[\Heis(\Sigma_+)],\C[\Heis(\Sigma_-)])$-bimodule, after identifying
the group $\Heis(-\Sigma_-)$  with $\Heis(\Sigma_-)^{op}$, and defining a right action of $\Heis(\Sigma_-)$ on $W_q(L_C)$   as the left action of the same element of $\Heis(-\Sigma_-)$. 
Then we can form the tensor product $W_q(L_C)\otimes_{\C[\Heis(\Sigma_-)]} W_q(L_-)$ and compare it with $W_q(L_+)$.

Any morphism in $3\mathrm{Cob}^{\rm{LC}}$ can be decomposed into simple ones, which are
mapping cylinders and  index $1$ or $2$ surgeries (see Section \ref{sec:app} for details).
We will first prove the isomorphism
$W_q(L_C)\otimes_{\C[\Heis(\Sigma_-)]} W_q(L_-)\cong
W_q(L_+)$
 and compute
  the 
induced maps
$ F_C: W_q(L_-)\to W_q(L_+)$ for simple cobordisms.
Then we
use Corollary \ref{app:cor} to
argue that the bimodule associated with a composition of simple cobordisms 
does not depend on the choice of the decomposition.


For a mapping cylinder $C_f:(\Sigma_-, L_-)\rightarrow (\Sigma_+, L_+)$, where 
the diffeomorphism $f:\Sigma_-\rightarrow \Sigma_+$ 
sends $L_-$ to $f_*(L_-)=L_+$, we have 
$${L}_{C_f}=\{(-x,f_*(x)),\ x\in H_1(\Sigma_-,\Z)\}$$ 
We choose a Lagrangian ${ L}_-^\vee$ 
complementary  to ${ L}_-$. Then ${ L}_+^\vee=f_*({ L}_-^\vee)$ is complementary to ${L}_+$. 
The submodule  $${ L}_{C_f}^\vee={ L}_-\oplus { L}_+^\vee\subset H_1(\partial C_f,\Z)$$ is Lagrangian and 
complementary to ${ L}_{C_f}$. Indeed,
 if $(-x,f_*(x))$ belongs to ${ L}_{C_f}^\vee$, then $x\in { L}_-$ and $f_*(x)\in { L}^\vee_+\cap {L}_+=\{0\}$,
showing that   ${ L}_{C_f}^\vee\cap { L}_{C_f}=\{0\}$. 
Recall that for all kinds of Lagrangians $L$, the notation ${\sf L}$ means $L \otimes \Z_{p'}$.
A $\C$-basis $b^+_y$  for 
$W_q(L_+)$ is labelled  by elements $y\in {\sf L}_+^\vee$.  {In all computations below we use the same notation for elements of $L^{\vee}$ as elements of $\Heis({\Sigma})$ acting on a module and ${\sf L}^{\vee}$ as the indexing set of a basis of $W_q(L)$. It shouldn't be confusing since in the second case they are just considered modulo $p'$.} Likewise for any element $(k,x)\in \widetilde{L}$ its image in $\widetilde{L}_p$ is just $(k \pmod{p}, x \pmod{p'})$, hence it acts again by $q^k$ on $\C_q$.

We have bases $\{B_z, z\in {\sf L}_{C_f}^\vee\}$ for $W_q(L_{C_f})$, and  $\{b^-_x, x\in {\sf L}_-^\vee\}$ for $W_q(L_-)$.
As a vector space the tensor product is generated by  
$$\{B_z\otimes b^-_x, z\in {\sf L}_{C_f}^\vee, x\in {\sf L}_-^\vee\}$$ 
with  relations
coming from the action by elements in $\Heis(\Sigma_-)$. We write $z\in {\sf L}_{C_f}^\vee$ as  $z=(z_-,z_+)$, $z_-\in {\sf L}_-$, $z_+\in {\sf L}_+^\vee=f_*({\sf L}_-^\vee)$. 

For an element $y\in {L}_-$ we get the relation
$$ q^{2y.x}B_{(z_-,z_+)}\otimes b^-_x
=B_{(z_-,z_+)}(0,y)\otimes b^-_x=
(0,(y,0))B_{(z_-,z_+)}\otimes b^-_x=
 B_{(z_-+y,z_+)}\otimes b^-_x$$
This reduces the set of generators to $\{B_{(0,z_+)}\otimes b^-_x, 
x\in {\sf L}^\vee_-,  z_+\in {\sf L}_+^\vee\}$.

For an element $x\in {L}_-^\vee$ we get another relation
\begin{align}\label{eq:mapcyl}
 B_{(0,z_+)}\otimes b^-_{x}=&B_{(0,z_+)}(0,x)\otimes \mathbf{1}=
(0, (x,0))B_{(0,z_+)}\otimes \mathbf{1}\\ \nonumber=&
(0,(0,f_*(x))(0,(x,-f_*(x)) B_{(0,z_+)}\otimes \mathbf{1}=
q^{ -2f_*(x).z_+} B_{(0,z_++f_*(x))}\otimes \mathbf{1}
\end{align}
where the intersection is written on the positively oriented $\Sigma_+$. 
This further reduces the genera\-tors to $\{B_{(0,z_+)}\otimes \mathbf{1}, z_+\in {\sf L}_+^\vee\}$.
Since any relation coming from any element in $\Heis(\Sigma_-)$ can be deduced from the previously written ones, 
 we get 
 that $\{B_{(0,y)}\otimes \mathbf{1}, y\in {\sf L}_+^\vee\}$ represents
 a $\C$-basis for the tensor product $W_q({L}_C)\otimes_{\C[\Heis(\Sigma_-)]}W_q(L_-)$.
 It follows that the $\C[\Heis(\Sigma_+)]$-module map 
 $$\psi_C: W_q(L_C)\otimes_{\C[\Heis(\Sigma_-)]} W_q(L_-)\rightarrow W_q(L_+)$$
  which sends $\mathbf 1\otimes \mathbf 1$ to $\mathbf 1$ is an isomorphism. Moreover, the map 
\begin{align*}
 F_C:W_q(L_-)&\rightarrow  W_q(L_C)\otimes_{\C[\Heis(\Sigma_-)]} W_q(L_-)
 \cong W_q(L_+)\\
b^-_x & \mapsto  \psi_C (\mathbf{1} \otimes b^-_x)
 \end{align*}   sends a basis vector $b^-_x$ to $b^+_{f_*(x)}$, for any $x\in {\sf L}_-^\vee$, by using \eqref{eq:mapcyl} with $z_+=0$.
 
 
 In the case of a simple cobordism $C:(\Sigma_-,L_-)\rightarrow (\Sigma_+,L_+)$ corresponding to an index $1$ surgery, the genus increases by $1$. We have 
 $$L_C=\{(-x_-,x_+),\ x_-\in H_1(\Sigma_-,\Z)\} \oplus \Z(0,\mu)$$ where
 $\mu$ is a meridian of the new handle and $x_+$ is the  class  $x_-$ pushed in $\Sigma_+$.
 Let $\lambda$ be a longitude for the new handle.  We choose a Lagrangian $L_-^\vee$ complementary to $L_-$. By pushing through the cobordism, we may also consider  $L_-^\vee$ as a subspace in $H_1(\Sigma_+,\Z)$. The span of $L_-^\vee$ and $\lambda$ gives a Lagrangian $L^\vee_+$ complementary to  $L_+$. 
 Then $L_C^\vee=L_-^\vee\oplus L_+^\vee$ is complementary to $L_C$ and the previous argument constructs the isomorphism. Here the map $$ F_C:W_q(L_-)\rightarrow W_q(L_C)\otimes_{\C[\Heis(\Sigma_-)]} W_q(L_-)
 \cong W_q(L_+)$$ sends a basis vector $b^-_x$ to $b^+_{x}$, where $x\in {\sf L}_-^\vee=L_-^\vee\otimes \Z_{p'}$.

 Let us consider a  simple cobordism $C:(\Sigma_-,L_-)\rightarrow (\Sigma_+,L_+)$ corresponding to an index $2$ surgery 
on a curve $\gamma$. 
Let $\delta$ be a curve in $\Sigma_-$
such that $\gamma. \delta=1$.
The curves $\gamma$ and $\delta$ determine a genus one subsurface $\Sigma_1$. Outside $\Sigma_1$ the cobordism is trivial.
 Denote by $\Sigma \subset \Sigma_-$ the complement of $\Sigma_1$ which we consider also  as a subsurface of $\Sigma_+$.
 We arrange the splitting so that $\Sigma_-=\Sigma\ \natural \ \Sigma_1$ is a boundary connected sum. Then all Lagrangian subspaces and Schr\"odinger modules split. Over $\Sigma$ the cobordism is trivial and the expected result is clear, so that it is enough to compute in the genus $1$ case, $\Sigma_-=\Sigma_1$ and $\Sigma_+=D^2$. The Lagrangian $L_-$ is generated by a simple curve $m$. A complementary Lagrangian $L_-^\vee$ is generated by $l$ with $m.l=1$.
 We have $\gamma=\alpha m+\beta l$, $\gcd(\alpha,\beta)=1$. 
The Lagrangian correspondence  is $$L_C= \Z(\gamma, 0)\quad\text{with complement}
\quad
L_C^\vee=\Z(\delta,0)$$ 
where
$\delta=um+vl$,  $\alpha v-\beta u=1$.
 Then $m=v\gamma - \beta \delta$, $l=-u\gamma+\alpha \delta$.
 We have bases  $B_{k\delta}$ and $b_{\nu l}$, $0\leq k,\nu<p'$ for  $W_q(L_C)$ and $W_q(L_-)$, respectively.
 Using $l$ we get the relation
$$ B_{k\delta}\otimes b_{(\nu+1) l}=B_{k\delta}(0,-u\gamma+\alpha \delta)
\otimes b_{\nu l}= (0,\alpha \delta)(u\alpha, -u\gamma) B_{k\delta}
\otimes b_{\nu l}=
q^{u \alpha+2ku}B_{(\alpha+k)\delta}\otimes b_{\nu l}\ $$
where we used 
intersection on $-\Sigma_-$.
This reduces the set of generators to $B_{k\delta}\otimes \mathbf{1}$, $0\leq k<p'$.
 The relation coming from  $m$ then gives
\begin{equation}
\label{eq:rel}
 B_{k\delta}\otimes \mathbf{1}=B_{k\delta}(0, v\gamma-\beta\delta)\otimes \mathbf{1}=
(0, -\beta\delta)(\beta v, v\gamma)B_{k\delta}\otimes \mathbf{1}=
q^{\beta v -2kv}B_{(k-\beta)\delta}\otimes \mathbf{1} 
.\end{equation}

If the surgery curve $\gamma$ is in $L_-$, we can choose $m=\gamma$, $l=\delta$.
The last relation gives in this case $B_{k\delta}\otimes \mathbf{1}=q^{-2k} B_{k\delta}\otimes \mathbf{1}$. Hence we have $B_{k\delta}\otimes \mathbf{1}=0$ for $0<k<p'$ and the tensor product is $\C$-generated by $\mathbf 1\otimes \mathbf 1$. The equalities
$$
B_{k\delta} \otimes \mathbf{1}=\begin{cases}
1& \text{if} \quad k\equiv 0 \pmod {p'}\\
0& \text{else}
\end{cases}
$$
 define an isomorphism $W_q(L_C)\otimes_{\C[\Heis(\Sigma_-)]}W_q(L_-)\cong \C_q$. In particular, 
$$ F_C(b_{k l})=\left\lbrace
 \begin{array}{ll}
1 &\text{ if}\quad k \equiv 0  \pmod {p'} \\
  0 &\text{ else}
 \end{array}
 \right.$$
If the surgery curve $\gamma$ is not in $L_-$ then $\beta\neq 0$. Let $d=\gcd(\beta,p')$, then the order of $\beta$ modulo $p'$ is $a=\frac{p'}{d}$. Hence,  relation \eqref{eq:rel} reduces the generators to
$\{B_{k\delta}\otimes \mathbf{1}, 0\leq k<d\}$.
Finally,  the action of $(0, am)_{-\Sigma_-}$ gives the following relation
$$B_{k\delta}\otimes \mathbf{1}=
(0, av\gamma-a\beta \delta) B_{k\delta}\otimes \mathbf{1}=
(0, -a\beta \delta)(a^2\beta v, av\gamma) B_{k\delta}\otimes \mathbf{1}=
q^{-2kav}B_{k\delta}\otimes \mathbf{1}$$
since $q^{a^2\beta}=1$ and the intersection pairing
is taken on $-\Sigma_-$.
 It follows $B_{k\delta}\otimes \mathbf{1}=0$ unless
$k$ is divisible by $d$, hence $\mathbf{1}\otimes \mathbf{1}$ generates
 $W_q(L_C)\otimes_{\C[\Heis_p(\Sigma_-)]}W(L_-)\cong \C_q$, as expected.

We are left with computing $ F_C$. The action of $(0, a\gamma)=
(0,a\alpha m+a\beta l)$ gives
$$\mathbf{1}\otimes b_{kl}=\mathbf{1}\otimes (0, a\beta l)(a^2\alpha\beta, a\alpha m) b_{kl}= q^{-2ka\alpha}\mathbf{1}\otimes b_{kl}$$
implying that
$ \mathbf{1}\otimes b_{kl}= 0$ if $d\nmid k$. If $d\,|\, k$ we set
$k\alpha=k'\beta$
and compute 
$$\mathbf{1}\otimes b_{kl}=\mathbf{1}(0, -ku\gamma+k\alpha\delta)\otimes 
\mathbf{1}=(0, k \alpha\delta)(k^2u\alpha, -ku\gamma)\mathbf{1}\otimes 
\mathbf{1}=
q^{k^2u\alpha}B_{k\alpha\delta}\otimes \mathbf{1}=q^{kk'\beta u}B_{k'\beta \delta }\otimes \mathbf{1}.
$$
From the action of $(0, k'm)$ we get
\begin{align*}
\mathbf{1}\otimes b_{kl}=&
q^{kk'\beta u} (0,k'v\gamma-k'\beta\delta) B_{k'\beta \delta }\otimes \mathbf{1}=
q^{kk'\beta u}(-(k')^2v\beta, k'v\gamma)(0, -k'\beta\delta)B_{k'\beta \delta }\otimes \mathbf{1} \\
=&
q^{kk'\beta u-kk'\alpha v}\mathbf{1}\otimes \mathbf{1}=
q^{-kk'}\mathbf{1}\otimes \mathbf{1}
=q^{-\alpha k^2/\beta}\mathbf{1}\otimes \mathbf{1}.
\end{align*}
We deduce that 
$$ F_C(b_{k l})=\left\lbrace
 \begin{array}{ll}
0 &\text{ if $d\nmid k$}\\
  q^{-\alpha k^2/\beta} &\text{ else.}
 \end{array}
 \right.$$
 
In order to complete the proof and
construct Schr\"odinger TQFT 
 we need to 
check the conditions in Corollary \ref{app:cor} for $ F$.
 

The first relation holds trivially since any $d\in\mathrm{Diff}(\Sigma)$, which is isotopic to identity,
 induces identity automorphism of $H_1(\Sigma)$ and as follows, identity automorphism of the corresponding Schr\"{o}dinger representation.

The second relation also holds since $d$ and $d^{\bS}$ can be extended to a diffeomorphism $M(\bS)\rightarrow M(\bS')$ which induces the relation.

{ If attaching and belt spheres $\bS$ and $\bS'$ of two index 1 or 2 surgeries do not intersect on a surface $\Sigma\in\mathbf{Surf}^{\rm LC}$, then it can be cut into two pieces $\Sigma=\Sigma_1\natural\Sigma_2$, such that $\bS\subset \Sigma_1$ and $\bS'\subset \Sigma_2$. It can be naturally extended to $M(\bS)$ and $M(\bS')$ so that all Lagrangians splits into direct sums, and Schr\"{o}dinger representations into tensor product (see discussion of monoidality below). It follows that $F_{M(\bS)}$ and $F_{M(\bS')}$ act independently on the corresponding factors, hence commute. It proves that relation 3 holds.}

Since definition of ${F}$ does not depend on the choice of the orientation on $\bS$,  relation $(5)$ holds.

The rest of the proof  deals with checking  that  ${F}$ preserves  relation $(4)$. For that, we compute Lagrangian correspondences   and the composition of two ${F}$ maps using some choice of bases.

Consider a composition of two cobordisms 
\begin{equation}
\label{eq:compos}
M(\bS_2)\circ M(\bS_1):(\Sigma_-, L_-)\rightarrow (\Sigma, L)\rightarrow (\Sigma_+, L_+)\end{equation}
of indices 1 and 2, such that the belt sphere $b(\bS_1)$  and the attaching $a(\bS_2)$ intersect at
  a single point. Denote by $L_{M(\bS_1)}$ and $L_{M(\bS_2)}$ the Lagrangian correspondences induced by these cobordisms. Let $\mu$ be the meridian of the handle attached in the first surgery (i.e. $\mu=b(\bS_1)$). Let $\gamma = a(\bS_2)$ be the attaching sphere of the second surgery. The composition \eqref{eq:compos} defines a diffeomorphism $\varphi:\Sigma_-\rightarrow \Sigma_+$ uniquely up to isotopy by the property $\varphi|_{\Sigma_-\cap\Sigma_+}=\id$ \cite{juhasz2018defining}. Then it induces the isomorphism $\varphi_*:H_1(\Sigma_-)\rightarrow H_1(\Sigma_+)$. 
    
Let's choose some orientation on $\gamma$ and $\mu$, 
such that $\mu.\gamma=1$.

The first homology group of $\Sigma$ can be represented as $H_1(\Sigma)\simeq H_1(\Sigma_-)\oplus \mathbb{Z}(\mu,\gamma)$ since $\mu.\gamma=1$. Let $i_0:H_1(\Sigma_-)\rightarrow H_1(\Sigma)$ be the corresponding embedding.

Since $L_{M(\bS_1)}=\{(x,-i_0(x))\; |\; x\in H_1(\Sigma_-)\}\oplus\mathbb{Z}(0,\mu)\subset H_1(-\Sigma_-)\oplus H_1(\Sigma)$ the image of $L_-$ in $H_1(\Sigma)$ equals to
$$L=L_{M(\bS_1)}.L_-=i_0(L_-)\oplus\mathbb{Z}(\mu) .$$

The manifold $M(\bS_2)$ is homotopy equivalent to $\Sigma \cup_{\partial D^2} D^2$, where $\partial D^2=S^1\rightarrow \Sigma$ is the inclusion of $\gamma$. Hence $H_1(M(\bS_2))\simeq H_1(\Sigma)/ [\gamma]$. 

Let $h_1 \simeq S^1 \times D^1$ be the handle glued in the first surgery and $h_2 \simeq D^2\times S^0$ -- in the second. Then $\varphi$ maps $(\Sigma_-\cap \mathrm{Im}(\bS_2))\cup \mathrm{Im}(\bS_1)\simeq D^2$ to $(h_1\backslash \mathrm{Im}(\bS_2))\cup h_2 \simeq D^2$ and is identical on $\Sigma\cap\Sigma_-$. 

Consider a cycle $x$ representing a class in $H_1(\Sigma_-)$. Let 
us first consider the case when $[x].\gamma=1$ and $x\cap \mathrm{Im}(\bS_2)=I_x\simeq [0,1]$. Then $\varphi|_{x\backslash I_x}=\id$ and $\varphi$ maps $I_x$ to an interval on $(h_1\backslash \mathrm{Im}(\bS_2))\cup h_2$ connecting two points on its boundary. It can be represented (up to homotopy) as $(\mu \backslash \mathrm{Im}\bS_2)\cup y\cup y'$ where $y$ and $y'$  are two chords -- one on each of two discs of $h_2$. This means that the class $\varphi_*(x)$ can be represented as $[x]-[\mu]$ in $M(\bS_2)$. If we consider now a class $[x]\in H_1(\Sigma_-)$ with an arbitrary intersection number $[x].\gamma$ an analogous construction gives $\varphi_*(x)=[x]-([x].\gamma)\mu $ in $H_1(M(\bS_2))$. Hence the image of $(i_0(x)+\gamma.x,-\varphi_*(x))\in H_1(\Sigma)\oplus H_1(\Sigma_+)$ is equal to zero under the homomorphism $H_1(\Sigma)\oplus H_1(\Sigma_+)\rightarrow H_1(M(\bS_2))$. Therefore

$$L_{M(\bS_2)}=\{(i_0(x)+(\gamma.x)\mu, -\varphi_*(x))| x\in H_1(\Sigma_-)\}\oplus\mathbb{Z}(\gamma,0) .$$

Let $L_-^{\vee}$ be a complement to $L_-$. Then  $L^{\vee}=i_0(L_-^{\vee})\oplus \mathbb{Z}(\lambda)$ is a complement to $L$ and
 $L_{M(\bS_2)}^{\vee}=\{(i_0(x),\varphi_*(y))|x\in L_-,\:y\in L_-^{\vee}\}\oplus  \mathbb{Z}(\mu,0)$ is a complementary Lagrangian to $L_{M(\bS_2)}$. Indeed, if $(x,y)\in L_-\oplus L_-^{\vee}$ and $(i_0(x)+k\mu,\varphi_*(y))\in L_{M(\bS_2)}$ then $x=-y = 0$, $k=0$ since $L_-\cap L_-^{\vee}=\{0\}$, so $L_{M(\bS_2)}\cap L_{M(\bS_2)}^{\vee}$=\{0\}. Restriction of the intersection pairing on $L_{M(\bS_2)}^{\vee}$ is equal to zero since it is a direct sum of two Lagrangians, and $L_{M(\bS_2)}+L_{M(\bS_2)}^{\vee}=H_1(\Sigma)\oplus H_1(\Sigma_+)$ since any element $(i_0(u)+n\mu+k\lambda, \varphi_*(v)), \: u,v\in H_1(\Sigma_-)$ can be decomposed as 
\begin{align}
\label{decomposition}
    (i_0(u)+n\mu+k\gamma, \varphi_*(v))& = \left(i_0(u_{L_-^{\vee}}-v_{L_-})+(\gamma.i_0(u_{L_-^{\vee}}-v_{L_-}))\mu + k\gamma, -\varphi_*(u_{L_-^{\vee}}-v_{L_-})\right)\\\nonumber
    &+
\left(i_0(u_{L_-}+v_{L_-})+(n-(\gamma.i_0(u_{L_-^{\vee}}-v_{L_-})))\mu,\varphi_*(u_{L_-^{\vee}}+v_{L_-^{\vee}})\right),
\end{align}
where $u=u_{L_-}+u_{L_-^{\vee}}, \:  v=v_{L_-}+v_{L_-^{\vee}}$ and $u_{L_-},v_{L_-}\in L_-,\: u_{L_-^{\vee}},v_{L_-^{\vee}}\in L_-^{\vee}$. 

Choose $L_+^{\vee}=\varphi_*(L_-^{\vee})$.
Then we can choose bases of Schr\"{o}dinger representations as follows:
\begin{align*}
W_q(L_-): &\; \{b_x^{-}|x\in \mathsf{L}_-^{\vee}\};\\
W_q(L): &\; \{b_{x+n\lambda}|x\in \mathsf{L}_-^{\vee},n\in\mathbb{Z}\};\\
W_q(L_+): & \;\{b^{+}_{x}|x\in\mathsf{L}_-^{\vee}\};\\
W_q(L_{M(\bS_2)}):& \;\{B_{(x+a\mu,y)}|(x,y)\in\mathsf{L}_-\oplus\mathsf{L}_-^{\vee},a\in\mathbb{Z}\}.
\end{align*}
The map $F_{M(\bS_1)}$ sends $b^{-}_x$ to $b_x$ then.
Consider $\mathbf{1}\otimes b_x \in W_q(L_{M(\bS_2)})\otimes W_q(L)$ for $x\in \mathsf{L}_-^{\vee}$. It can be rewritten using \eqref{decomposition} as follows:
\begin{align*}
\mathbf{1}\otimes b_x & =\mathbf{1}\otimes (0,i_0(x))\mathbf{1}=(0,(i_0(x),0))\mathbf{1}\otimes \mathbf{1}\\
&=(0,(-(\gamma.i_0(x))\mu,\varphi_*(x)))(0,(i_0(x)+(\gamma.i_0(x))\mu),-\varphi_*(x))\mathbf{1}\otimes \mathbf{1}\\ &=
(0,(0,\varphi_*(x)))\mathbf{1}\otimes (0,-(\gamma.i_0(x))\mu)\mathbf{1}=(0,(0,\varphi_*(x)))\mathbf{1}\otimes\mathbf{1},
\end{align*}
which means that the image of $b_x$ is $b_{\varphi_*(x)}$ and it coincides with the action of the mapping cylinder $C_{\varphi}$ associated to $\varphi$.


{Finally, we have to prove that $F$ preserves the monoidal structure. The monoidal {product} 
\newline $\Sigma_1\natural\Sigma_2$ of two surfaces $\Sigma_1,\Sigma_2\in\CobL$ with one boundary component is induced by boundary connected sum $\partial \Sigma_1 \#\partial\Sigma_2\simeq S^1$. Since $H_1(\Sigma_1\natural\Sigma_2)=H_1(\Sigma_1)\oplus H_1(\Sigma_2)$, the new Lagrangian is $L=L_1\oplus L_2$. It means that the Heisenberg group is the direct product over the center of the two Heisenberg groups associated to $(\Sigma_1,L_1)$ and $(\Sigma_2,L_2)$. It induces the isomorphsm $W_q(L)\simeq W_q(L_1)\otimes W_q(L_2)$. Similarly, one can check that $ F$ preserves the monoidal structure on morphisms.}
\end{proof}

It remains to compare $F$ with the abelian TQFT. 
 
\begin{proof}[{\bf Proof of Theorem \ref{thm:iso-tqft}}]

 We will use the skein model from Section 2.
Following \cite[Theorem 4.5]{Gelca_Uribe} 
 the Heisenberg group algebra $\C[\Heis(\Sigma)]$ can be identified with the $U(1)$-skein algebra $S( \Sigma)$. 
This makes the TQFT vector space $V(\Sigma,L)$  to a module over $\C[\Heis(\Sigma)]$. 
Actually,   it is isomorphic to the Schr\"odinger representation
(see \cite[Theorem 4.7]{Gelca_Uribe} for $p$ even).

Here we construct  the isomorphism explicitly.
Let us denote by  $S_p( \Sigma)$ the reduced 
$U(1)$ skein algebra,
where $q$ is specified to the $p$th root of unity and a $p'$
copies  of any curve are removed
\cite[def 4.3]{Gelca_Uribe}.
 Then $S_p( \Sigma)$
is identified 
with  $\C[\Heis_p(\Sigma)]$ by sending a simple closed curve $\gamma$ with blackboard framing  to the class of  the image of
$$(0,[\gamma])\in \Z\times H_1(\Sigma,\Z)=\Heis(\Sigma)$$
in $\Heis_p(\Sigma)$. 
Let $H$ be a handlebody with boundary $\Sigma$ such that ${ L}$
 is the kernel of the inclusion $H_1(\Sigma,\Z)\hookrightarrow H_1(H,\Z_{})$.
 Then the TQFT vector space $V(\Sigma,L)$ is the quotient of $S_p(\Sigma)$ by  the subspace generated by 
 $\gamma-1$ where $\gamma$ is a 
 simple curve  that bounds in $H$ or equivalently such that $[\gamma]\in  L$. Using the isomorphism
 $S( \Sigma)\cong \C[\Heis(\Sigma)]$, we deduce that the quotients $V(\Sigma,L)$  and $W_q(L)$ are isomorphic. 
 
A basis $\{b_x, x\in {\sf L}^\vee\}$ 
for $W_q(L)$
can be
 represented by skein elements $\{y_x, x\in {\sf L}^\vee\}$ in $H$ providing a basis for $V(\Sigma, L)$.
 Here for  an embedded curve $x$ in $\Sigma$, the element $y_x$ 
 is obtained by pushing $x$ in $H$ with blackboard framing
 and then by taking its skein class. For example,
the element  $y_{3x}$ correspond to the three parallel copies of $y_x$
obtained by using the blackboard framing.  
 We are now able to compare $\check F_C= Z(\check C)F_C$ with the TQFT map  on simple cobordisms.
 
 Let us consider a mapping cylinder $C_f:(\Sigma_-,L_-)\rightarrow (\Sigma_+,L_+)$ with $g_-=g_+=g$. A basis for the TQFT vector space
 (identified with the Schr\"odinger representation $W_q(L_-)$)
 is represented by a handlebody $H_-$, with $\partial H_-=\Sigma_-$ and with the cores $l_i$, $1\leq i\leq g$,
 of its handles colored by $y^k$, $0\leq k\leq p'$.
 The TQFT map is represented by gluing  the mapping  cylinder  $C_f$ to
  the handlebody $H_-$. This results in a handlebody $H_+$ with boundary $\Sigma_+$. Moreover, when pushing  the colored curve $l$ across the cylinder we get a curve parallel to $f(l)$. Hence, 
the TQFT map sends $y^k$ in $H_-$  to $f(y^k)$ in $H_+$,
matching  
  $F_{C_f}$. Note that $\check C_f$ is an integral homology connected sum of $g$
  copies of $S^2\times S^1$, since $f$ preserves the Lagrangians.
  Hence, in our normalization
  $Z(\check C_f)=1$.

 In the case of an index $1$ surgery $C:(\Sigma_-,L_-)\rightarrow (\Sigma_+,L_+) $, the TQFT map is represented by the inclusion of a handlebody $H_-\hookrightarrow H_+=H_-\cup_{\Sigma_-}C$, where $\partial H_-=\Sigma_-$ and $$\ker\left(H_1(\Sigma,\Z)\hookrightarrow H_1(H_-,\Z)\right)={ L}_-.$$
 This inclusion map matches again $F_{C}$ with $Z(\check C)=1$.
 
 In the case of an index $2$ surgery on a curve $\gamma$, we only need to consider the case where $\Sigma_-$ is a genus $1$ surface.
 Then the TQFT map $Z(C): V(\Sigma_-, L_-)\rightarrow \C_q$ 
 is given by the
 evaluation of the skein element
$(H_-,x)$ inside  $M_\gamma=(H_-\cup_{\Sigma_-} C)\cup_{ S^2}  D^3$.
If $\gamma=m$, $M_\gamma= S^1\times  S^2$
 and the evaluation 
  reduces to  a Hopf link with one Kirby-colored component, which is zero unless $x=0$, when it is $1$. 
 
If $\gamma=l$, $M_\gamma= S^3$ and the evaluation is $1$ for all  $x=y^k$.
Hence, in both cases we recover $F_C$.

More generally,  
 for $\gamma=\alpha m+\beta l$ with  $\beta\neq 0$, the manifold $M_\gamma$ is the lens space
 $L(\beta, \alpha)$. 
Let us choose a  continued fraction decomposition
$\beta/\alpha=[m_1,\dots, m_n]$ as in \cite{Li-Li}.
Then a surgery link $L$ for $M_\gamma$ is  the 
 length $n$ Hopf chain with framings $m_i$.
 Hence, the TQFT map sends $y^k$ to the following number
$$ 
Z(C)(y^k)=\kappa^{-\text{\rm sign}(L)}\; \eta^n\sum^{p'}_{j_1,\dots,j_n=1}\; q^{\sum^n_{i=1} m_i j^2_i}\;
 q^{2kj_1}\; q^{2 \sum^{n-1}_{i=1}j_i j_{i+1}}
 $$
 \begin{figure}
\centering
\begin{picture}(300,80)
\put(0,0){\includegraphics[scale=0.6]{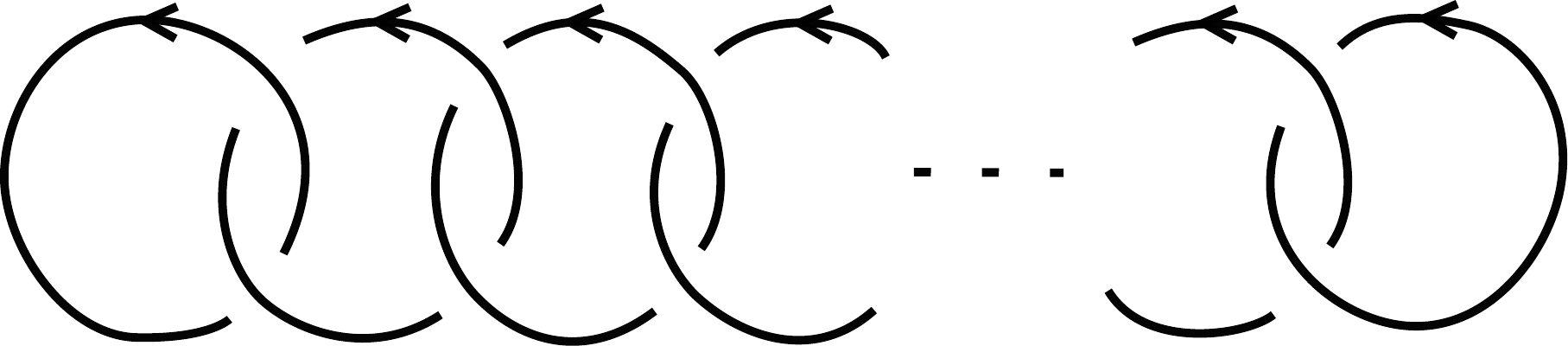}}
\put(18,8){$y^k$}
\put(70,8){$\eta\omega$}
\put(110,8){$\eta\omega$}
\put(150,8){$\eta\omega$}
\put(240,8){$\eta\omega$}
\put(280,8){$\eta\omega$}
\put(20,72){$0$}
\put(70,72){$m_1$}
\put(110,72){$m_2$}
\put(150,72){$m_3$}
\put(240,72){$m_{n-1}$}
\put(280,72){$m_n$}
\end{picture}
\caption{Surgery link for the lens space where the upper indices correspond to the framings  and the lower ones to the colors.}
\label{lens}
\end{figure}



Since  a  recursive computation of this sum was done in \cite{Li-Li},
we present here just the result. 
\[Z(C)(y^k)=\begin{cases} 0 & d\nmid k\\
q^{-\frac{\alpha k^2}{\beta}} Z(L(\beta, \alpha))& \text{else}
\end{cases}
 \]
where $d=\text{\rm gcd}(\beta, p')$. This  coincides
with $F_C$ on this cobordism with $Z(\check C)=Z(L(\beta, \alpha))$.

Since  any cobordism is  a composition of  simple ones and  $F_C$ is functorial,  for any cobordism $C$
we have the TQFT map 
$$Z(C): V(\Sigma_-)\cong W_q(L_-)\rightarrow  V(\Sigma_+)\cong W_q(L_+)$$
which is equal, up to a coefficient, to the inclusion map
$W_q(L_-)\rightarrow W_q(L_C)\otimes_{\C[\Heis_p(\Sigma)]} W_q(L_-)$ composed with the isomorphism from Theorem \ref{thm:main}. Closing with handlebodies compatible with the Lagrangians we get that the coefficient is $Z(\check C)$ which completes the proof.
\end{proof}


\section{Schr\"odinger local systems on surface configurations}
\label{Sec4}
In this section we construct and study two projective representation of $\mathrm{Mod}(\Sigma)$ on Schr\"odinger representations.
Note that the results of this section are independent from the rest of the paper.

\subsection{Heisenberg group as a quotient of the surface braid group}\label{sec.4.1}

Let $\Sigma$ be an oriented surface of  genus $g$ with one boundary component. For $n\geq 2$, the unordered configuration space\index{configuration space} of $n$ points in $\Sigma$ is
\[
\mathrm{Conf}_{n}(\Sigma )= \{ \{c_{1},\dots,c_{n}\} \subset \Sigma \mid c_i\neq c_j \text{ for $i\neq j$}\}.
\]
 The surface braid group\index{surface braid group}  is then defined as ${B}_{n}(\Sigma)=\pi_{1}(\mathrm{Conf}_{n}(\Sigma),*)$. 
To construct a presentation, 
we fix based loops, $\alpha_1,\dots,\alpha_g,\beta_1,\dots,\beta_g$ on $\Sigma$, as depicted in Figure \ref{modelSurf}.
The base point $*_1$ on $\Sigma$ belongs to the base configuration $*$
in $\mathrm{Conf}_{n}(\Sigma)$. By abuse of notation, we use $\alpha_r$, $\beta_s$ also for the loops in $\mathrm{Conf}_{n}(\Sigma)$
 where only the point at $*_1$ is moving along the corresponding curve. We write composition of loops  from right to left.
\begin{figure}
\centering
\includegraphics[scale=0.6]{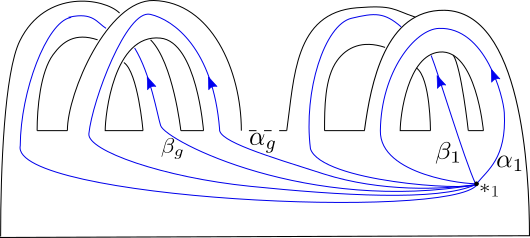}
\caption{Model for $\Sigma$.}
\label{modelSurf}
\end{figure}
The braid group ${B}_{n}(\Sigma)$ has generators $\alpha_1,\ldots,\alpha_g$, $\beta_1,\ldots,\beta_g$ 
together with the classical braid generators $\sigma_1,\ldots,\sigma_{n-1}$,  and relations:
\begin{equation}
\label{eq:relations}
\begin{cases}
 [\sigma_{i},\sigma_{j}] = 1 & \text{for } \lvert i-j \rvert \geq 2, \\
\sigma_{i}\sigma_{j}\sigma_{i}=\sigma_{j}\sigma_{i}\sigma_{j} & \text{for } \lvert i-j \rvert = 1, \\
 [\zeta,\sigma_{i}] = 1 & \text{for } i>1 \text{ and all } \zeta\text{ among the }\alpha_r,\beta_s, \\
 [\zeta,\sigma_{1}\zeta\sigma_{1}] = 1 & \text{for all } \zeta\text{ among the }\alpha_r,\beta_s, \\
 [\zeta,\sigma^{-1}_{1}\eta\sigma_{1}] = 1& \text{for all } \zeta\neq \eta \text{ among the }\alpha_r,\beta_s, \text{ with}\\
& \{\zeta,\eta\}\neq \{\alpha_r,\beta_r\},\\
\sigma_{1}\beta_{r}\sigma_{1}\alpha_{r}\sigma_{1}=\alpha_{r}\sigma_{1}\beta_{r} & \text{for all } r.
\end{cases}
\end{equation}

  We denote by $x.y$  the standard intersection form on $H_1(\Sigma,\Z)$. The Heisenberg group\index{Heisenberg group}  $\mathcal{H}(\Sigma)$ is the central extension of the homology group $H_1(\Sigma,\Z)$  by the intersection $2$-cocycle $(x,y)\mapsto x.y$. As a set 
 $\mathcal{H}(\Sigma)$ is equal to
  $\Z \times H_1(\Sigma,\Z)$, with  the group structure 
\begin{equation}
\label{eq:Heisenberg-product}
(k,x)(l,y)=(k+l+\,x.y,x+y).
\end{equation}
 
 We  use the notation $a_r$, $b_s$ for the homology classes of $\alpha_r$, $\beta_s$, respectively. Let us denote by
   $[\sigma_1,{B}_n(\Sigma)]$ the normal subgroup 
of the surface braid group $B_n(\Sigma)$   
    generated by the commutators $\{[\sigma_1,x]$, $x\in {B}_n(\Sigma)\}$.
 From the presentation above we obtain the following (see \cite{HeisenbergHomology} for more details).
\begin{proposition}
\label{hom_phi}
 For each $g\geq 0$ and $n\geq  2$, the quotient 
 $${B}_n(\Sigma)/[\sigma_1,{B}_n(\Sigma)]\xrightarrow{\sim} \mathcal{H}(\Sigma) $$ 
is isomorphic 
to the Heisenberg group. 
 An isomorphism 
 is  induced by the surjective homomorphism
\[
\phi \colon {B}_{n}(\Sigma) \relbar\joinrel\rightarrow \Heis(\Sigma)
\]
sending each $\sigma_i$ to $u=(1,0)$, $\alpha_r$ to $\tilde{a}_r=(0,a_r)$, $\beta_s$ to $\tilde{b}_s=(0,b_s)$. \\
\end{proposition}
It follows that any representation of the Heisenberg group $\Heis(\Sigma)$ is also a representation of the surface braid group ${B}_n(\Sigma)=\pi_{1}(\mathrm{Conf}_{n}(\Sigma),*)$ and hence provides a local system on the configuration space $\mathrm{Conf}_{n}(\Sigma)$.

Let us denote by $\mathrm{Aut}^+(\mathcal{\Heis}(\Sigma))$ the group of automorphisms of $\Heis(\Sigma)$ acting by identity on the center $\langle u\rangle=\Z\times 0$, namely an element of this group sends $(n,x)$ to $(c(x)+n,l(x))$ for some $l\in \mathrm{Sp}(H_1(\Sigma))$ and $c\in H^1(\Sigma)$. 
By \cite[Lemma 15]{HeisenbergHomology}
we have the following  split short exact sequence
 \begin{equation*}
   1\rightarrow H^1(\Sigma,\Z)\xrightarrow{j} \mathrm{Aut}^+(\Heis(\Sigma)) \xrightarrow{l} \mathrm{Sp}(H_1(\Sigma))\to 1
   \end{equation*}
    where $j(c)=[(k,x)\to (k+c(x),x)]$ and $\mathrm{Sp}(H_1(\Sigma))$ is the symplectic group 
    preserving the intersection pairing.
 The homomorphism $l$ has a section  
\begin{equation}
\label{eq:section}
s: g\mapsto [(k,x)\mapsto (k,g(x))]
\end{equation}  
  providing a semi-direct decomposition $\mathrm{Aut}^+(\Heis(\Sigma)) \cong \mathrm{Sp}(H_1(\Sigma)) \ltimes H^1(\Sigma;\Z)$.

Let us denote by $\mathrm{Mod}(\Sigma)$ the mapping class group.
Its action on $H_1(\Sigma,\Z)$ preserves the symplectic form, and hence
using the section $s$ from \eqref{eq:section} 
we get a {\em symplectic action} of the mapping class group on the Heisenberg group, where $f\in \mathrm{Mod}(\Sigma)$ acts by
\begin{equation}\label{eq:symaction}
(k,x)\mapsto (k,f_*(x)).
\end{equation} 
On the other hand,  the quotient map $\phi: B_n(\Sigma)\to \Heis(\Sigma)$ induces a different action of  $\mathrm{Mod}(\Sigma)$ on $\Heis(\Sigma)$.
 The following proposition is  proved in \cite[Section 3]{HeisenbergHomology}.
 \begin{proposition}
\label{f_Heisenberg}
For  $f\in \mathrm{Mod}(\Sigma)$, there exists a unique homomorphism $f_{\Heis} \colon \Heis(\Sigma) \rightarrow \Heis(\Sigma)$ such that the following square commutes:
\begin{equation}
\label{eq:projection-equivariance}
  \begin{tikzcd}
     B_{n}(\Sigma) \arrow[d,swap, "\phi"] \arrow[rr, "f_{B_{n}(\Sigma)}"] && B_{n}(\Sigma) \arrow[d,"\phi"]\\
     \Heis(\Sigma) \arrow[rr, "f_{\Heis}"] && \Heis(\Sigma)
  \end{tikzcd}
\end{equation}
\end{proposition}
We obtain an action of $\mathrm{Mod}(\Sigma)$ on the Heisenberg group $\Heis(\Sigma)$ given by
\begin{align}
\label{eq:action_on_Heis}
 \mathrm{Mod}(\Sigma) &\longrightarrow  \mathrm{Aut}^+(\Heis(\Sigma))\\
 f & \mapsto  f_\Heis: (k,x)\mapsto (k+\theta_f(x),f_*(x)) \nonumber
\end{align}
where the map $\theta: \mathrm{Mod}(\Sigma) \to H^1(\Sigma,\Z)$ sending $f$ to $\theta_f \in \mathrm{Hom}(H_1(\Sigma),\Z)$ { satisfies the  {\em crossed homomorphism} property $\theta_{g\circ f}(x)=\theta_f(x)+\theta_g(f_*(x))$.
Clearly, both actions coincide on $\mathrm{Sp}(H_1(\Sigma))$, i.e.
  $l(f_\Heis)=f_*$. However, $f_*$ and $f_\Heis$ are different on $\Heis(\Sigma)$. 
  
  It is shown in \cite[Proposition 19]{HeisenbergHomology} that the crossed homomorphism $\theta$ is given by a formula due to Morita \cite[Section 6]{Morita1989}, which is stated  as follows.
For $\gamma\in \pi_1(\Sigma)$, let us denote by $\gamma_i$ the element in the free group generated by $\alpha_i$, $\beta_i$ that is the image of $\gamma$ under the homomorphism that maps the other generators to $1$. Then 
we have a  decomposition
\[
\gamma_i=\alpha_i^{\nu_1}\beta_i^{\mu_1}\dots \alpha_i^{\nu_m}\beta_i^{\mu_m}\ ,
\]
where $\nu_j$ and $\mu_j$ are $0$, $-1$ or $1$. The integer $d_i(\gamma)$ is then defined by
\begin{align*}
d_i(\gamma) &= \sum_{j=1}^m\nu_j\sum_{k=j}^m\mu_k-\sum_{j=1}^m\mu_j\sum_{k=j+1}^m\nu_k \\
&= \sum_{j=1}^m \sum_{k=1}^m \iota_{jk} \nu_j \mu_k,
\end{align*}
where $\iota_{jk} = +1$ when $j\leq k$ and $\iota_{jk} = -1$ when $j > k$.
The formula for the crossed homomorphism  is
\begin{equation*}
\mathfrak{\theta}(f)([\gamma])=\sum_{i=1}^g d_i(f_\sharp(\gamma))-d_i(\gamma)\ .
\end{equation*}
 }

\subsection{A variant of Weil representation\label{sec3.2}}
Recall that for a lagrangian subspace $L$ and $q$ a $p$-th root of unity, we defined the Schr\" odinger representation $W_q(L)$ of the finite quotient $\Heis_p(\Sigma)$ of the Heisenberg group.
The following finite dimensional version of the famous Stone--von~Neumann theorem holds. 
\begin{theorem}[Stone--von~Neumann]
For $q$ a root of unity of order $p$,    $W_q(L)$ is  the unique irreducible unitary representation of $\Heis_p(\Sigma)$, up to unitary isomorphism, where the central generator $u=(1,0)$ acts by $q$.
\end{theorem}
A proof for even $p$ can be found in \cite[Theorem 2.4]{Gelca_Uribe}. The odd case works similarly.

For the rest of this section
we fix an  {\bf odd} integer $p\geq 3$.
Denote by $\rho: \Heis_p\rightarrow GL(W_q(L))$ the Schr\"oginger representation. For any automorphism $\tau\in \mathrm{Aut}^+(\Heis_p)$, we have another representation $\rho\circ\tau: \Heis_p\rightarrow GL(W_q(L))$ called  {\em twisted} representation of $\Heis_p$  and denoted by ${}_\tau\!W_q(L)$. From the Stone--von~Neumann Theorem we get an isomorphism of representation $W_q(L)\approx {}_\tau\!W_q(L)$ which is well defined up to a scalar in $S^1$.

For odd $p$, the action  of $\mathrm{Mod}(\Sigma)$ on the Heisenberg group from Proposition \ref{f_Heisenberg}
passes to the finite quotient, and we denote by $f_{\Heis_p}$
the resulting automorphism for every $f\in \mathrm{Mod}(\Sigma)$.

Hence, for a mapping class $f$, we obtain
a unitary isomorphism
$\mathcal{S}_\Heis(f): W_q(L)\xrightarrow{\sim} {}_{f_{\Heis_p}}\!W_q(L)$ 
defined by the following  commutative diagram 
\[
\begin{tikzcd}
W_q(L) \ar[rr,"{\mathcal S_\Heis(f)}"] \ar[d,"{\rho(k,x)}",swap]  &&{}_{f_{\Heis_p}}\!W_q(L)\ar[d,"{\rho(f_{\Heis_p}(k,x))}"]\\
W_q(L) \ar[rr,"{\mathcal S_\Heis(f)}"] &&{}_{f_{\Heis_p}}\!W_q(L) .
\end{tikzcd}
  \]
Applying the Stone-von Neumann Theorem,
this provides a homomorphism 
 $$ \mathcal{S}_\Heis: \mathrm{Mod}(\Sigma)\rightarrow \mathrm{PU}(W_q), \quad
 \text{where} \quad \mathrm{PU}(W_q)=\mathrm{U}(W_q(L))/\mathbb S^1$$ is the projective unitary group which does not depend on the choice of $L$.

Denote by ${}_{f_*}\!W_q(L)$ the Schr\"odinger representation twisted by the symplectic action, we also have an isomorphism
$\mathcal{S}(f): W_q(L)\xrightarrow{\sim} {}_{f_*}\!W_q(L)$
defined, up to a  scalar 
in $\mathbb S^1\subset \mathbb C$, by the condition
\begin{equation}\label{eq:intertwinS}
\rho(k,f_*(x))\circ \mathcal{S}(f)= \mathcal{S}(f)\circ \rho(k,x) ,\ \text{for any}\ \ (k,x)\in \Heis_p(\Sigma).
\end{equation}
This provides another homomorphism
$$ \mathcal{S}: \mathrm{Mod}(\Sigma)\rightarrow \mathrm{PU}(W_q)$$
The next result 
was   proven independently  by Gelca with collaborators \cite[Theorem 8.1]{Gelca_Uribe}, \cite{Gelca_Hamilton}, \cite[Chapter 7]{Gelca}.
\begin{corollary}\label{thm:Gelca}
The homomorphism $\mathcal{S}:\mathrm{Mod}(\Sigma)\rightarrow \mathrm{PU}(W_q)$
given by the symplectic action  
is isomorphic to the one resulting from the abelian TQFT on
$\mathrm{3Cob}$ described in Section 2. 
\end{corollary}



In general, any projective representation of a group $G$ can be linearised on an appropriate central extension. 
Given an homomorphism $R: G\rightarrow PGL(V)$, where $V$ is a complex vector space,  a choice of lift (as a set map)
$\tilde R: G\rightarrow GL(V)$ defines a defect map $c: G\times G\rightarrow \C^*$,
by $\tilde R(gg')=c(g,g')\tilde R(g)\tilde R(g')$.
In the case of a projective unitary representation the map $c$ takes values in $\mathbb{S}^1$.
It is well known from  basic group cohomology theory that $c$ is a $2$-cocycle defining a central extension of $G$ on which
$R$ can be linearised. This central extension is classified by the class $[c]\in H^2(G,\C^*)$. If this  class can be reduced to a subgroup, then the linearisation already arises on a smaller extension. If $[c]=0$, the minimal extension  is $G$ itself.

Projective actions of $\mathrm{Mod}(\Sigma)$ on
Schr\"odinger representations  are naturally  equipped with
such cohomology classes, determined by the Stone--von Neumann 
isomorphisms. 
As explained in Section \ref{sec:LC}, the homomorphism  $\mathcal{S}$ can be linearised and we use the same notation for a linearisation 
$\mathcal{S}: \mathrm{Mod}(\Sigma)\rightarrow U(W_q(L))$. 
We will show that the extension which linearises the projective representation $\mathcal{S}_\Heis$
is non trivial by computing its classifying 2-cocycle.


A key observation is that, for a mapping class $f$, the automorphism
$f_{\Heis_p}: \Heis_{p}(\Sigma)\rightarrow \Heis_{p}(\Sigma)$ is equal to the symplectic one composed with an inner automorphism
$$f_{\Heis_p}(k,x)=(k+\theta_f(x),f_*(x))=(0,f_*(t_f))(k,f_*(x))(0,-f_*(t_f)), $$
where $2t_f\in H_1(\Sigma,\Z_p)$ is the Poincar\'e dual of $\theta_f$, i.e. $\theta_f (x) = 2t_f.x$. Here we use that $2$ is invertible modulo $p$ and that $f_*(t_f).f_*(x)=t_f.x$.
Acting on $W_q(L)$ we get 
the following commutative diagram
\[
\begin{tikzcd}
W_q(L) \ar[rr,"\mathcal{S}(f)"] \ar[d,"{\rho(k,x)}",swap] &
& {}_{f_*}\!W_q(L) \ar[rr,"{\rho(0,f_*(t_f))}"]\ar[d,"{\rho(k,f_*(x))}"] &&{}_{f_{\Heis_p}}\!W_q(L)\ar[d,"{\rho(f_{\Heis_p}(k,x))}"]\\
W_q(L) \ar[rr,"\mathcal{S}(f)"]&& {}_{f_*}\!W_q(L)\ar[rr,"{\rho(0,f_*(t_f))}"]&&{}_{f_{\Heis_p}}\!W_q(L)
\end{tikzcd}
\]
Hence  the two projective  actions
are related as follows 
$$\mathcal{S}_\Heis(f)=  \rho(0,f_*(t_f))\circ \mathcal{S}(f)=\mathcal{S}(f)\circ \rho(0,t_f)\ .$$

We can now compute the cocycle from the intertwining isomorphism
$$W_q(L)\cong {}_{(fg)_\Heis}\!W_q(L)={}_{g_\Heis}\!\!\left({}_{f_\Heis}\!W_q(L)\right)\ .$$
\begin{align*}
\mathcal{S}_\Heis(f) \circ \mathcal{S}_\Heis(g)&= \mathcal{S}(f)\circ  \rho(0,t_f) \circ \mathcal{S}(g)\circ  \rho(0,t_g)\\
&= \mathcal{S}(f)\circ \mathcal{S}(g)\circ \rho(0,g_*^{-1}(t_f))\circ   \rho(0,t_g)
\\
&=    \mathcal{S}(f)\circ \mathcal{S}(g)\circ \rho(g_*^{-1}(t_f).t_g,g_*^{-1}(t_f)+t_g)\\
&=q^{g_*^{-1}(t_f).t_g }   \mathcal{S}_\Heis(fg).
\end{align*}
Here we used that the crossed homomorphism property $\theta_{fg}=\theta_g+g^*(\theta_f)$
implies for  the Poincar\'e dual 
$t_{fg}=t_g+g_*^{-1}(t_f)$.
Using $t_{gg^{-1}}=t_{g^{-1}}+g_*(t_g)=0$,
we get that the cocycle  is equal to $q^{c(f,g)}$ where $c(f,g)= g_*^{-1}(t_f).t_g=t_f.g_*(t_g)=-t_f.t_{g^{-1}}$.

Morita studied in \cite{Morita1997} the intersection cocycle $(f,g)\mapsto c_{\mathrm{Mor}}(f,g)=t_{f^{-1}}.t_{g}=c(g,f)$ which represents $12c_1$ where $c_1$ is the Chern class  generating $H^2(\mathrm{Mod}(\Sigma),\Z)=\Z$ for surfaces  of genus  at least $3$. The Meyer cocycle $\tau(f,g)$ is the signature of the oriented $4$-dimensional manifold defined as the surface bundle over the pair of pants with monodromy $f$ and $g$ on $2$ boundary components. This definition is symmetric in $f$ and $g$ so that we have $\tau(f,g)=\tau(g,f)$. From Morita work we have that $[c_{\mathrm{Mor}}]=3[\tau]=12 c_1$. By switching the variable we get $[c]=3[\tau]=12 c_1$.
It follows that for odd $p$ the projective action $\mathcal{S}_\Heis: \mathrm{Mod}(\Sigma)\rightarrow PU(W_q)$ cannot be linearised on the mapping class group while the symplectic action does.

By Corollary \ref{thm:Gelca},  abelian and hence Schr\"odinger TQFT 
reproduces  symplectic action 
$$ \mathcal{S}: \mathrm{Mod}(\Sigma)\rightarrow \mathrm{PU}(W_q)$$
after restricting to mapping classes preserving  Lagrangians. 
We just argued that $\mathcal S_\Heis$ is essentially different from $\mathcal S$.
This leads us
to the following problem.

\begin{problem}
Construct a TQFT on $3\mathrm{Cob}^{\rm{LC}}$ using Schr\"odinger representations
that recovers the action $\mathcal{S}_{\Heis}$ on mapping cylinders.
\end{problem}

\section{Appendix}\label{sec:app}

For the reader's convenience we 
adapt here the Juh\'{a}sz presentation of  cobordism categories from \cite{juhasz2018defining} to $\mathrm{3Cob}^{\rm LC}$. 

Let us define the category $\mathbf{Surf}^{\rm LC}$, 
which is an analogue of $\mathbf{Man}_2$ in \cite{juhasz2018defining},
as follows.
An object of $\mathbf{Surf}^{\rm LC}$
is an oriented compact surface $\Sigma$ with $S^1$-parametrized boundary $S^{1}\xrightarrow{\simeq}\partial \Sigma$ equipped with a Lagrangian $L\subset H_1(\Sigma;\mathbb{Z})$. A morphism from $(\Sigma_-,L_-)$ to $(\Sigma_+,L_+)$ is a diffeomorphism $d:\Sigma_-\rightarrow \Sigma_+$ preserving the boundary and  Lagrangians, so that $d_*(L_-)=L_+$.

A framed sphere in $(\Sigma,L)\in \mathbf{Surf}^{\rm LC}$ is a smooth orientation--preserving embedding $\bS:S^{k}\times D^{2-k}\hookrightarrow \Sigma$ for $k=0,1$ such that $\mathrm{Im}(\bS)\cap \partial \Sigma=\varnothing$. By performing surgery on $\Sigma$ along $\bS$ we construct 
$$\Sigma(\bS)=(\Sigma\backslash \mathrm{Im}(\bS))\cup_{S^k\times S^{1-k}} D^{k+1}\times S^{1-k}$$ 
 for $k=0,1$. The trace of the surgery
 $$M(\bS)=(\Sigma\times I)\cup_{\mathrm{Im}(\bS)\times\{1\}} D^{k+1}\times D^{2-k}$$ is a 3-cobordism from $\Sigma$ to $\Sigma(\bS)$ obtained by attaching the handle $D^{k+1}\times D^{2-k}$ to $\Sigma \times I$
 with the Lagrangian $$L_{M(\bS)}:=\mathrm{Ker} \left(i_*: H_1(\Sigma)\oplus H_1(\Sigma(\bS))\to H_1(M(\bS))\right).$$ Let us now define the induced Lagrangian in $H_1(\Sigma(\bS);\mathbb{Z})$ as  
 $L_{M(\bS)}.L$.
We conclude that the surgery along $\bS$ induces a morphism in $\mathrm{3Cob}^{\rm LC}$ between $(\Sigma,L)$ and $(\Sigma(\bS),L_{M(\bS)}.L)$ which we  denote by $(M(\bS), L_{M(\bS)})$.

Let $\mathcal{G}^{\rm LC}$ be the directed graph obtained  from the category $\mathbf{Surf}^{\rm LC}$ by adding an edge $e_{\Sigma,\bS}$ from $(\Sigma,L)$ to  $(\Sigma(\bS),L_{M(\bS)}.L)$  for every $(\Sigma,L)$ and  $\bS$ in $\Sigma$.

Let $\mathcal{F}(\mathcal{G}^{\rm LC})$ be the free category generated by $\mathcal{G}^{\rm LC}$.

\begin{definition}\cite[Definition 1.4]{juhasz2018defining} The set of relations $\mathcal{R}$ in $\mathcal{F}(\mathcal{G}^{\rm LC})$ 
is defined as follows:
\begin{enumerate}
 \item $e_d\circ e_{d'}=e_{d\circ d'}$, where $e_d,e_{d'}$ and $e_{d\circ d'}$ are edges corresponding to diffeomorphisms. If $d$ is isotopic to the identity, then $e_d=\id_\Sigma$. Also $e_{\Sigma,\varnothing}=\id$.
    \item For an orientation preserving diffeomorphism $d:\Sigma\rightarrow \Sigma'$ sending the Lagrangian $L$ to $L'$ and for a framed sphere $\bS\subset \Sigma$, let $\bS' =d\circ \bS$ and $d^{\bS}$ be the induced diffeomorphism, then $e_{\Sigma,\bS'}\circ e_{d}= e_{d^{\bS}}\circ e_{\Sigma,\bS}$
    \item If $\bS$ and $\bS'$ are two disjoint framed spheres in $\Sigma$, $(\Sigma,L)\in \mathrm{3Cob}^{\rm LC}$, then $e_{\Sigma,\bS}$ and $e_{\Sigma,\bS'}$ commute.
    \item If $\bS'\subset \Sigma(\bS)$ is a framed sphere of index 1, $\bS$ is a framed sphere of index 2 and the attaching sphere $a(\bS')$ intersects the belt sphere $b(\bS)$  transversely at one point.  Then $e_{\Sigma(\bS),\bS'}\circ e_{\Sigma,\bS}= e_{\varphi}$,
    where $\varphi:\Sigma\to \Sigma(\bS)(\bS')$ is a diffeomorphism which is identical on $\Sigma\cap\Sigma(\bS)(\bS')$ and unique up to isotopy (see details in \cite[Definition 2.17]{juhasz2018defining}).
    \item  $e_{\Sigma,\bS}= e_{\Sigma,\overline{\bS}}$, where $\overline{\bS}$ is the same sphere with the opposite orientation.
\end{enumerate}
\end{definition}

Define the functor $P:\mathcal{F}(\mathcal{G}^{\rm LC})\rightarrow \mathrm{3Cob}^{\rm LC}$ by sending vertices to itself, diffeomorphisms to mapping cylinders and the edges $e_{M,\bS}$ to cobordisms $M(\bS)$. The aim of this section is to prove that $P$ descends to a functor on $\mathcal{F}(\mathcal{G}^{\rm LC})/{\mathcal{ R}}$, which is an isomorphism of categories.

\vspace{0.5cm}

\subsection{Morse data and Cerf decompositions}

Given two objects  $(\Sigma_{-},L_-),(\Sigma_{+},L_{+})$ in $3\Cob^{\rm LC}$ and $[M]\in 3\Cob^{\rm LC}((\Sigma_{-},L_-),(\Sigma_{+},L_{+}))$, where $L_-\subset H_1(\Sigma_-)$, $L_+\subset H_1(\Sigma_{+})$ and $L_M\subset H_1(\Sigma_-)\oplus H_1(\Sigma_+)$ are the corresponding Lagrangians with $L_+=L_M.L_-$. 
Here we write $[M]$ to emphasise that we refer to an equivalence class of 3-cobordisms. Recall that two 3-cobordisms are equivalent if they are diffeomorphic and the restriction of the diffeomorphism to the boundary respects
the parametrizations (compare \cite[Def. 2.2]{juhasz2018defining}).

Let us choose a representative $M$ of the equivalence class $[M]$.  The boundary $\partial M$ is decomposed as 
$$\partial M=\partial_-(M)\cup \partial_0(M)\cup \partial_+(M)$$ where $\partial_\pm(M)$ is identified with $\Sigma_\pm$ and $\partial_0(M)$ is parametrised by $I\times S^1$.

\begin{definition} A Morse datum for $M$ consists of a pair $(f,\mathbf{b}, v)$ of a Morse function on $M$, an ordered tuple $\mathbf{b}=(0=b_0<b_1<\dots<b_m=1)\subset \R$ and a gradient-like field $v$ for $f$ such that:
\begin{itemize}
\item {  on $I\times S^1$, $f$ coincides with the first coordinate map $(t,x)\mapsto t$  and $v=\frac{\partial}{\partial t}$;}
    \item $\Sigma_-=f^{-1}(b_0)$ and $\Sigma_+=f^{-1}(b_m)$ are the sets of minima and maxima of $f$;
    \item each $f^{-1}(b)$ is connected, and $f$ has no critical points of index 0 and 3;
    \item $f$ has different values at critical points;
    \item all critical points belong to the interior of $M$;
    \item $b_1,\dots,b_{m-1}$ are regular values of $f$, such that each $(b_i,b_{i+1})$
     contains at most one critical point.
\end{itemize}
\end{definition}

The Lemma 2-1 in \cite{borodzik2016morse} can be applied here with $Y=\partial_0(M)$ proving that Morse functions with the required boundary conditions exist on any morphism $M$ in $3\Cob^{\rm LC}$. Using that Morse functions  with different values at critical points are generic, in particular dense in the space of functions the proof given there show existence of a Morse function with the required boundary conditions and distinct values at critical points. We obtain a gradient-like vector field by using a Riemannian metric (cf Lemma 1-6) in \cite{borodzik2016morse}). We first choose a collar of $\partial_0(M)$, $]-1,0]\times I \times S^1\subset M$ on which we put the product Riemannian metric $g_1$. Then we take any Riemannian metric $g_2$ on $M\setminus ]-\frac{1}{2},0]\times I \times S^1$. Using a partition of unity we build a metric $g$ on $M$ which coincides with $g_1$ on $[-\frac{1}{2},0]$ and $g_2$ on $M\setminus ]-1,0]\times I \times S^1$. Using $g$ we get a gradient vector field $\nabla f$ for the Morse function $f$ and obtain a Morse datum (compare with Lemma 1.7 in \cite{borodzik2016morse}).



 We call a 3-cobordism $M$ {\it simple} if  it admits a Morse function with at most one critical point. A Cerf decomposition of $M$ is a presentation of $M$ as  a composition of simple cobordisms. We will need here a refined
 notion.

\begin{definition} A {\it parametrized Cerf decomposition} of $(M, L_M)$ 
consists of the following data:
\begin{itemize}
\item
a Cerf decomposition $$(M,L_M)=(M_m,L_{M_m})\circ (M_{m-1},L_{M_{m-1}})\circ\dots\circ (M_1,L_{M_1})$$ 
where  $(M_i, L_{M_i}):(\Sigma_{i-1}, L_{i-1})\to
(\Sigma_{i}, L_{i})$ are simple cobordisms,
$L_0=L_-$ and the other $L_i$ are defined inductively by
$L_{i}:=L_{M_i}.L_{i-1}$ with
 a $L_{M_i}=\Ker( H_1(\Sigma_{i-1})\oplus H_1(\Sigma_{i})\rightarrow H_1(M_i))$;
\item
a framed attaching sphere $\bS_i\in \Sigma_i$ and a diffeomorphism $D_i: M(\bS_i)\to M_i$, such that
$D_i|_{\Sigma_i(\bS_i)}: \Sigma_i(\bS_i)\to \Sigma_{i+1}$ is an orientation preserving diffeomorphism, well-defined up to isotopy, and
$D_i|_{\Sigma_{i-1}}=\id$.
\end{itemize}

\label{app:cerf}
\end{definition}


\begin{proposition} Each simple cobordism $M$ is either diffeomorphic to a mapping cylinder or to a 3-cobordism $M(\bS)$ obtained by index 1 or 2 surgery. 
\end{proposition}
\begin{proof}
A simple cobordism $M$ admits a Morse datum $(f,v)$ with at most one critical point of index $k=1,2$. In the case with no critical point, the flow defines a trivialisation
$M\cong I\times \Sigma_-$ which is a mapping cylinder. In the case of a critical point $p$, the  stable and unstable manifolds $M^s(p),\: M^{u}(p)$ of the critical point $p$ do not intersect $\partial_0(M)$, then the intersection of $M^s(p)$ with $\Sigma_-$ defines
a  $(k-1)$-sphere ($k$ is the index) which can be framed, giving a framed sphere $\bS$ and a diffeomorphism 
$\Sigma_-(\bS)\cong M$ well defined up to isotopy \cite[Section 2.2]{juhasz2018defining}.
%
\end{proof}

More generally, for any cobordism $M$ in $3\Cob^{\rm LC}$ a Morse datum  $(f,\mathbf{b},v)$ induces a parametrized Cerf decomposition.
Indeed, simple  cobordisms  between level sets $\Sigma_{i}=f^{-1}(b_i)=M_{i}\cap M_{i+1}$ are
$M_i=f^{-1}([b_{i-1},b_{i}])$. 
Furthermore, the flow of $v$ defines the induced attaching sphere on each level set as the intersection with the stable manifold $M^s(p)$ of the  corresponding critical point if it exists. If not, this flow   defines a diffeomorphism between $\Sigma_i$ and $\Sigma_{i+1}$
 which is identity on the boundary.

For the reader convenience let us show
 that Cerf decompositions preserve Lagrangian correspondences.

\begin{lemma} Assume $(\Sigma_i, L_i)$ are objects of $3\Cob^{\rm LC}$ for $0\leq i\leq m$. For a Cerf decomposition $$(M,L_M)=(M_m,L_{M_m})\circ (M_{m-1},L_{M_{m-1}})\circ\dots\circ (M_1,L_{M_1})$$ 
we have  $L_{m}=L_M.L_0$.
\label{lem: decomp}
\end{lemma}
\begin{proof} 

It suffices to show that for two  morphisms $M_1:(\Sigma_-,L_-)\rightarrow (\Sigma, L)$ and $M_2:(\Sigma,L)\rightarrow (\Sigma_+,L_+)$ from $\CobL$, their composition is also a morphism from $3\Cob^{\rm LC}$, i.e. 
\begin{equation}\label{Lagrangian_property}
    L_{M_2}.(L_{M_1}.L_-)=L_{M_2\circ M_1}.L_-.
\end{equation}
Since both sides of \eqref{Lagrangian_property} are Lagrangians, it's enough to show that any $x\in L_{M_2}.(L_{M_1}.L_-)$ also belongs to $L_{M_2\circ M_1}.L_-$. By the definition of Lagrangian correspondence $\exists y\in H_1(\Sigma)$ and $\exists z\in L_{\Sigma_-}$, such that $(x,y)\in L_{M_2}$ and $(y,z)\in L_{M_1}$. Hence $(x,z)\in L_{M_2\circ M_{1}}$, since the image of $(x,z)$ in $H_1(M_2\circ M_1)$ can be rewritten as $x+z=y-y=0$.
\end{proof}

\begin{definition} A {\it diffeomorphism equivalence} of 
parametrized 
Cerf decompositions $$(M_m,L_{M_m})\circ\dots\circ (M_1,L_{M_1}) \quad \text{and} 
\quad
(M'_m,L_{M'_m}) \circ\dots\circ (M'_1,L_{M'_1})$$ of the same length is a collection of diffeomorphisms $\phi_i:M_i\rightarrow M'_i$ such that 
\begin{itemize}
    \item they are identities on the ends: $\phi_1|_{\Sigma_-}=\id_{\Sigma_-}$ and $\phi_m|_{\Sigma_+}=\id_{\Sigma_+}$;
    \item they are compatible: $\phi_i|_{\Sigma_{i}}=\phi_{i+1}|_{\Sigma_{i}}$;
    \item they preserve Lagrangians: $(\phi_i|_{\Sigma_{i}})_*(L_{i})=L'_{i}$.
\end{itemize}
\end{definition}


As explained in \cite{juhasz2018defining} (before Thm. 2.24) diffeomorphism equivalences are induced by left-right equivalence of Morse data.
Let us recall here the definition of the latter for completeness.
Two Morse data $(f,\mathbf{b},v)$ and $(f',\mathbf{b}',v')$ are related by a {\it left-right equivalence} if there are diffeomorphisms $\Phi:M\rightarrow M$ and $\phi:\R\rightarrow\R$, such that $f'=\phi\circ f\circ\Phi^{-1}$, $\mathbf{b}'=\phi(\mathbf{b})$, $v'=\Phi_*(v)$, $\Phi|_{\Sigma_\pm}$ are isotopic to $\id$.
Then for a given Cerf decomposition we can obtain  
a diffeomorphism equivalent 
one by  
 setting $M'_i:=\Phi(M_i)$.

\begin{theorem}\label{thm:Cerf-app}
 Any two parametrised Cerf decompositions of $(M,L_M)$ are connected by sequence of the following moves (and their inverse):
\begin{enumerate}
    \item critical point cancellation: the composition of two simple morphisms $(M_i,L_{M_i})$ and $(M_{i+1},L_{M_{i+1}})$ of indices 1 and 2, whose belt and attaching spheres intersect transversally in one point, are replaced by the cylindrical morphism 
    $(\Sigma_{i-1}\times [0,1], L_{\Sigma_{i-1}\times [0,1]})$;
    \item critical point crossing: two simple morphisms $(M_i,L_{M_i})$ and $(M_{i+1},L_{M_{i+1}})$ of indices $l$ and $k$, whose belt and attaching spheres do not intersect (for some choice of metric), are replaced by the pair of morphisms $(M'_i,L_{M'_i})$ and $(M'_{i+1},L_{M'_{i+1}})$ of indices $k$ and $l$, such that $(M_{i+1}\circ M_i, L_{M_{i+1}\circ M_i})\simeq(M'_{i+1}\circ M'_i, L_{M'_{i+1}\circ M'_i})$;
     \item cylinder gluing: two simple cobordisms $(M_i,L_{M_i})$ and $(M_{i+1},L_{M_{i+1}})$, one of which is a mapping cylinder, are replaced by the composition $(M_{i+1}\circ M_{i}, L_{M_{i+1}\circ M_{i}})$;
     \item diffeomorphism equivalences.
\end{enumerate}
\end{theorem}
\begin{proof}For each cobordism $M$ in $3\Cob^{LC}$ we construct an associated cobordism $\tilde{M}$ between closed surfaces by gluing a cylinder $D^2\times I$ along the parameterised boundary $\partial_0(M)$. Then we extend both Cerf decompositions to it and apply  \cite[Theorem 1.7]{juhasz2018defining} within its usual setting. In order to restrict back to surfaces with boundaries we have to make sure that all the moves can be performed in such way that it doesn't affect the glued cylinder.
\begin{enumerate}
\item[(1,2)] The critical point cancellation and critical point crossing moves can be obtained by an isotopy supported in a tubular neighborhood of stable and unstable manifolds \cite{cerf1970stratification}. Since all belt and attaching spheres belong to the interior of $M$,  we can always choose this neighborhood to be inside $\mathrm{Int}(M)$. Thus, this isotopy can be restricted to $M$.
    \item[(3)] Let $\tilde{M}_i$ be a mapping cylinder associated with a diffeomorphism $d:\tilde{\Sigma}_{i-1}\rightarrow \tilde{\Sigma}_i$, then the composition of this cobordism with $\tilde M(\bS)$
    is diffeomorphic to $\tilde{M}(\bS'):\tilde{\Sigma}_i\rightarrow \tilde{\Sigma}_i(\bS')$, where 
    $\bS'=d\circ\bS$. Since diffeomorphism $d$ is identical outside $\Sigma_i$ in $\tilde{\Sigma}_i$,  the induced diffeomorphism of cobordisms is identical on $D^2\times I$  and can be restricted.
    \item[(4,5)] Moves 4 and 5 are induced by isotopies which are identical on $D^2\times I$, hence can be restricted to $M$.
\end{enumerate}

Since each pair of Cerf decompositions is connected by a move relating diffeomorphic cobordisms, it preserves Lagrangian correspondences.
\end{proof}

\begin{theorem}  The functor $P:\mathcal{F}(\mathcal{G}^{\rm LC})\to \CobL$ descends to a functor $\mathcal{F}(\mathcal{G}^{\rm LC})/\mathcal{R}\xrightarrow{\simeq} \mathrm{3Cob}^{\rm LC}$ which is an isomorphism of  categories.
\end{theorem}

\begin{proof}
First we should check that relations $\mathcal R$ hold in $\CobL$. Relation $(1)$ in $\mathcal R$ can be realized by moves $(3)$ and $(4)$ in Theorem \ref{thm:Cerf-app}. Relation $(2)$ follows from  move $(4)$. Relation $(3)$ follows from the critical point crossing move. Relation $(4)$ also  follows from the critical point cancellation move. The last relation follows from the fact that $M(\bS)=M(\overline{\bS})$.

Now we prove that $P$ is a bijection on the hom-sets. If we have a morphism $(M,L)\in\CobL$, then it admits a Morse datum, and hence as discussed above  the induced parametrized Cerf decomposition $$(M,L)=(M_m,L_{M_m})\circ (M_{m-1},L_{M_{m-1}})\circ\dots\circ (M_1,L_{M_1})$$ 
with $M_i$ in $\CobL$ for $1\leq i\leq m$ by Lemma \ref{lem: decomp}.
Then assigning to each simple morphism $(M_{i},L_{M_{i}})$ an edge in $\mathcal{G}^{\rm LC}$ ($e_{\Sigma,\bS}$ to surgeries and  $e_d$ to mapping cylinders), we obtain a preimage of $(M,L)$ in $\mathcal{G}^{\rm LC}$. This proves that $P$ is surjective onto morphisms of $\CobL$.

Assume we have  two morphisms $f,f'$ in $\mathcal{F}(\mathcal{G})/\mathcal{R}$ which give the same morphism in $\CobL$. Each of them can be represented by some composition of edges in $\mathcal{G}$, inducing natural Cerf decompositions on $P(f)$ and $P(f')$. Then these decompositions are connected by a sequence of moves listed in Theorem \ref{thm:Cerf-app}. Since each move corresponds to a relation in $\mathcal{R}$ the morphisms $f$ and $f'$ belong to the same equivalence class in $\mathcal{F}(\mathcal{G}^{\rm LC})/\mathcal{R}$. This proves that $P$ is injective.
\end{proof}

As a corollary, we have the following construction of TQFTs on $\CobL$.

\begin{corollary} \label{app:cor}
Assume a map $F:\mathcal{G}^{\rm LC}\rightarrow \mathrm{Vect}_{\C}$ (sending vertices to vector spaces and arrows to linear maps)  satisfies the following properties:
\begin{enumerate}
    \item $F(e_{\Sigma,\varnothing})=\id$ and if $d$ is isotopic to the identity, then $F(d)=\id$.
    \item For an orientation preserving diffeomorphism $d:\Sigma\rightarrow \Sigma'$ sending the Lagrangian $L$ to $L'$ and a framed sphere $\bS\subset \Sigma$, let $\bS' =d\circ \bS$ and $d^{\bS}$ be the induced diffeomorphism, then $F(e_{\Sigma,\bS'})\circ F(e_{d})= F(e_{d^{\bS}})\circ F(e_{\Sigma,\bS})$
    \item If $\bS$ and $\bS'$ are two disjoint framed spheres in $\Sigma$, $(\Sigma,L)\in \mathrm{3Cob}^{\rm LC}$, then $F(e_{\Sigma,\bS})$ and $F(e_{\Sigma,\bS'})$ commute.
    \item If $\bS'\subset \Sigma(\bS)$ is a framed sphere of index 1, $\bS$ is a framed sphere of index 0 and the attaching sphere $a(\bS')$ intersects the belt sphere $b(\bS)$  transversely in one point. Then there is a diffeomorphism $\varphi: \Sigma\rightarrow \Sigma(\bS)(\bS')$ for which $F(e_{\Sigma(\bS),\bS'})\circ F(e_{\Sigma,\bS})= F(e_{\varphi})$
    \item  $F(e_{\Sigma,\bS})= F(e_{\Sigma,\overline{\bS}})$, where $\overline{\bS}$ is the same sphere with the opposite orientation.
\end{enumerate}
Then $F$ descends to the unique  functor $F:\CobL\rightarrow \mathrm{Vect}_{\C}$.
\end{corollary}
\vspace{0.5cm}

\bibliographystyle{plain}
\bibliography{biblio}

\end{document}